\documentclass{siamart190516}

\usepackage{lipsum}
\usepackage{amsfonts}
\usepackage{graphicx}
\usepackage{epstopdf}
\usepackage{algorithmic}
\usepackage{rotating}

\ifpdf
  \DeclareGraphicsExtensions{.eps,.pdf,.png,.jpg}
\else
  \DeclareGraphicsExtensions{.eps}
\fi

\usepackage{enumitem}
\setlist[enumerate]{leftmargin=.5in}
\setlist[itemize]{leftmargin=.5in}

\newsiamremark{remark}{Remark}
\newsiamremark{comment}{Comment}
\newsiamremark{hypothesis}{Hypothesis}
\crefname{hypothesis}{Hypothesis}{Hypotheses}
\newsiamthm{claim}{Claim}

\headers{On the iterative solution of systems of the form {$A^\top A\lowercase{x}=A^\top\lowercase{b}+\lowercase{c}$}}{H.~Calandra, S.~Gratton, E.~Riccietti, X.~Vasseur}

\title{On the iterative solution of systems of the form $A^\top A\lowercase{x}=A^\top\lowercase{b}+\lowercase{c}$\thanks{Submitted to the editors: Draft: Version of \today.
\funding{This work was funded by TOTAL.}}}

\author{Henri Calandra\thanks{TOTAL, Centre Scientifique et Technique Jean F\'eger, avenue de Larribau F-64000 Pau, France 
  (\email{henri.calandra@total.com}).}
\and Serge Gratton\thanks{INPT-IRIT, University of Toulouse and ENSEEIHT, 2 Rue Camichel, BP~7122, F-31071 Toulouse Cedex~7, France
  (\email{serge.gratton@enseeiht.fr}, \email{elisa.riccietti@enseeiht.fr}).}
\and Elisa Riccietti\footnotemark[3]
\and Xavier Vasseur\thanks{ISAE-SUPAERO, University of Toulouse, 10, avenue Edouard Belin, BP~ 54032, F-31055 Toulouse Cedex 4, France
	(\email{xavier.vasseur@isae-supaero.fr}).}
}

\usepackage{amsopn}

\newcommand{\Ae}{A_\epsilon}
\newcommand{\be}{b_\epsilon}

\newcommand{\tr}{\tilde{r}}

\newcommand{\tx}{\tilde{x}}
\newcommand{\ata}{(A^TA)^{-1}}
\newcommand{\apa}{A^\dagger(A^\dagger)^T}
\newcommand{\LM}{Levenberg-Marquardt }
\newcommand{\xe}{x_\epsilon}

\ifpdf
\hypersetup{
  pdftitle={On the iterative solution of systems of the form {$A^\top A\lowercase{x}=A^\top\lowercase{b}+\lowercase{c}$}},
  pdfauthor={H.~Calandra, S.~Gratton, E.~Riccietti, X.~Vasseur}
}
\fi

\begin{document}

\maketitle

\begin{abstract}
  Given a full column rank matrix $A \in \mathbb{R}^{m\times n}$ ($m\geq n$), we consider a special class of linear systems of the form $A^\top Ax=A^\top b+c$ with $x, c \in \mathbb{R}^{n}$ and $b \in \mathbb{R}^{m}$.  The occurrence of $c$ in the right-hand side of the equation prevents the direct application of standard methods for least squares problems. Hence, we investigate alternative solution methods that, as in the case of normal equations, take advantage of the peculiar structure of the system to avoid unstable computations, such as forming $A^\top A$ explicitly. We propose two iterative methods that are based on specific reformulations of the problem and we provide explicit closed formulas for the structured condition number related to each problem. These formula allow us to compute a more accurate estimate of the forward error than the standard one used for generic linear systems, that does not take into account the structure of the perturbations. The relevance of our estimates is shown on a set of synthetic test problems. Numerical experiments highlight both the increased robustness and accuracy of the proposed methods compared to the standard conjugate gradient method. It is also found that the new methods can compare to standard direct methods in terms of solution accuracy.
\end{abstract}

\begin{keywords}
  Linear systems, Conjugate Gradient Method, Forward Error, Least Squares Problems.
\end{keywords}

\begin{AMS}
  15A06, 65F10, 65F35, 65G50.
\end{AMS}

\section{Introduction}
Given $A\in \mathbb{R}^{m\times n}$, $m\geq n$ with $\mathrm{rank}(A)=n$, $b\in \mathbb{R}^{m}$ and $x, c\in\mathbb{R}^n$, we consider the following linear system of equations:
\begin{equation}\label{base}
A^\top Ax=A^\top b+c,
\end{equation}
or, equivalently, the following minimization problem:
\begin{equation}\label{qls}
\min_{x \in \mathbb{R}^n} \frac{1}{2}\|Ax-b\|^2-c^\top x.
\end{equation} 
At first sight, relation \cref{base} evokes a least squares problem in the normal equations form:
\begin{equation}\label{normal_eq}
A^\top Ax=A^\top b,
\end{equation}
but the situation is fundamentally different, as the vector $c$ is present in the right-hand side. Contrarily to least squares problems, a very well-studied topic (\cite{bjorck1996numerical}\cite[Ch.5]{vanloan}\cite[Ch.10]{nw}), problem \cref{base} has not been object of much study in the literature. It can be seen as a reformulation of a particular instance of a KKT system \cite{chen2012structured,Duff,gill1993solving,sun1999structured}:
\begin{equation}\label{augmented}
\begin{bmatrix}
H & A\\
A^\top& 0
\end{bmatrix}
\begin{bmatrix}
r\\x
\end{bmatrix}=\begin{bmatrix}
b\\-c
\end{bmatrix}
\end{equation} 
where $H\in\mathbb{R}^{m\times m}$ is a symmetric matrix, in our case  $H$ is the identity matrix. We mention \cite{bjor92} where the author studies the numerical solution of this augmented system by Gaussian elimination and a proper scaling. 

The solution of problem \cref{base} is however required in  various applications in optimization, such as multilevel Levenberg-Marquardt methods \cite{multilevel_ann}, or in certain formulations based on penalty function approaches \cite[\S 7.2]{fletcher1973,orban}, that we describe in  \cref{sec_motiv}.
Motivated by the important applications these problems arise from, we believe that it is worth getting more insights on both theoretical and practical aspects of the solution of \cref{base}, in particular on its  solution by iterative methods. 

Due to the similarity with the normal equations, we expect to encounter, both in the theoretical study of \cref{base} and in the numerical solution, issues similar to those reported in the literature on normal equations. 

First of all, the system matrix is the same. It is well known that the product $A^\top A$ should not be explicitly formed, because the accuracy attainable by methods for the solution of the normal equations may be much lower than that attainable by a backward stable method for least squares. In this case indeed, the best forward error bound for the normal equations can be obtained by the classical sensitivity analysis of linear systems. It is then of the order of  $\kappa^2(A)~u$, with $u$ the machine precision and $\kappa(A)=\|A\|\|A^\dagger\|$ the condition number of $A$ in the Euclidean norm \cite{higham2002accuracy,rice}. This is an underwhelming result, as from Wedin's theorem (\cite[Theorem 20.1]{higham2002accuracy}) the sensitivity of a least squares problem is measured by $\kappa^2(A)$ only when the residual is large, and it is measured by $\kappa(A)$ otherwise. 

However, practical solution methods do not form the product $A^\top A$, and the following key observation is rather exploited.   The peculiar structure of the normal equations allows to write:
\begin{equation*}
A^\top Ax-A^\top b=A^\top(Ax-b), 
\end{equation*} 
that makes it possible to either employ a factorization of $A$ rather than of $A^\top A$ (in the case of direct methods) or to perform matrix-vector multiplications of the form $A^\top x$ and $Ax$ rather than $A^\top Ax$ (in the case of iterative methods).
This fact makes the standard conditioning analysis of linear systems unsuitable to predict the error for such methods, as the linear system  is not subject to normwise perturbations on the matrix $A^\top A$, but rather to structured perturbations, such as perturbations in the matrix $A$ only. A structured conditioning analysis is then more relevant and indeed leads to the same conclusions as for least squares problems \cite{gratton1996condition}. 

It is therefore possible to devise stable implementations of methods for normal equations. This has been deeply investigated in the literature, see, e.g., \cite{bjorck1998stability,vanloan,paige1982lsqr}. In \cite{bjorck1998stability} the specialized implementation of the Conjugate Gradient (CG) method for the normal equations (CGLS) has been proposed and shown to be more stable than CG applied directly to \cref{normal_eq}. 


In general, however, not all the considerations made for normal equations directly apply to \cref{base}, since the right-hand side in \cref{base} is $A^\top b+c$ instead of $A^\top b$. This difference has important consequences both on theoretical and practical sides. 

On one hand, this results in a different mapping for the condition number and a different set of admissible perturbations for the backward error, as \cite{Fraysse2000}. Consequently, the existing perturbation theory for least squares problems \cite{arioli,gratton1996condition,walden} does not apply here. A proper analysis of the structured condition number of the problem should then be developed.

From a practical perspective, even though the system matrix is the same, a fundamental difference appears: $A$ cannot be directly factorized on the right-hand side and standard methods for the normal equations cannot be directly applied. Successful algorithmic procedures used for normal equations could however be adapted to design stable solution methods, tailored for the specific problem that we consider. 

\paragraph{Contributions} Inspired by the existing results on least squares problems, we propose two different reformulations of \cref{base}, such that it is possible to factorize the matrix $A$ on both the left and right-hand sides, and two corresponding solution methods. We provide structured condition numbers for \cref{base} and for the proposed reformulations. These allow us to compute  first order estimates of the forward error on the solution computed by the proposed methods.  We also report on the numerical performance of such methods on a relevant set of test problems, considering also direct methods to provide a fair comparison. Extensive numerical experiments confirm an improved stability of the proposed methods compared to CG.  The estimates of the forward error are also validated numerically and shown to be sharper upper bounds than the standard bounds  issued from the theory of linear systems.



\paragraph{Structure} The manuscript is organized as follows. In  \cref{sec_motiv} we present two applications arising in optimization where the solution of \cref{base} is required. In  \cref{sec_teo}, we report on the conditioning of the problem and backward error analysis. In \cref{sec_error_analysis}, we discuss issues related to the solution of system of the form \cref{base} in finite arithmetic. We remind  specific algorithmic procedures successful for least squares problems (case $c=0$) that we later adapt to design stable methods for the solution of \cref{base} (more general case $c\neq 0$).  Then, we introduce and study the new iterative solution methods in  \cref{sec_methods}. In  \cref{sec_forward} we propose first order estimates of the forward error on the solution computed by the proposed methods. Extensive numerical experiments are proposed in  \cref{sec_num}. Conclusions are drawn in  \cref{sec:conc}.

\paragraph{Notation}\label{sec_not}

Given a matrix $A\in\mathbb{R}^{m\times n}$, we denote by $A^\dagger$ its pseudoinverse, by $\kappa(A)=\|A\|\|A^\dagger\|$ the condition number of $A$ ($\|\cdot\|$ being the Euclidean norm) and by $\sigma_{\min}(A),\sigma_{\max}(A)$ the smallest and largest singular values of $A$, respectively. 
The Frobenius norm of $A$ is defined as 
\begin{equation*}
\|A\|_F:=\left(\sum_{i,j}\lvert A_{i,j}\rvert^2\right)^{1/2}=tr(A^\top A)^{1/2}.  
\end{equation*}
Moreover, given  $b \in \mathbb{R}^{m}, c \in \mathbb{R}^{n}$, we define $\|[A,  b,  c]\|_F:=\sqrt{\|A\|_F^2+\|b\|^2+\|c\|^2}$.  
We later denote by $I_n\in\mathbb{R}^{n\times n}$ the identity matrix of order $n$, $\otimes$ the Kronecker product of two matrices and by $\mathrm{vec}$ the operator that stacks the columns of a matrix into a vector of appropriate 
dimension \cite{gratton1996condition}. Given $\mathcal{X}$ and $\mathcal{Y}$ two normed subspaces equipped with norms $\|\cdot\|_{\mathcal{X}}$ and $\|\cdot\|_{\mathcal{Y}}$, respectively, we denote by $\|\cdot\|_{\mathrm{op}}$ the operator norm induced by the choice of $\|\cdot\|_{\mathcal{X}}$ and $\|\cdot\|_{\mathcal{Y}}$. Finally, given $a \in \mathbb{R}$, we denote by 
$\mathrm{fl}(a)$ its floating point representation, i.e. $\mathrm{fl}(a)$ is such that
$\lvert a-\mathrm{fl}(a)\rvert \leq \lvert a\rvert\, u $, where $u$ represents the machine precision. 

\section{Two motivating applications}\label{sec_motiv}
In this section, we describe two different applications in which problems of the form \cref{qls} arise, motivating our interest in their solution.  

\subsection{Fletcher's exact penalty function approach}
The first applicative context arises in equality constrained minimization. Consider a problem of the form:
\begin{align*}
\min_x f(x)\\
\text{ s.t. }g(x)=0,
\end{align*}
for twice differentiable functions $f:\mathbb{R}^n\rightarrow \mathbb{R}$ and $g:\mathbb{R}^n\rightarrow \mathbb{R}^m$, respectively. The solution of systems of the form \cref{base} is needed for the evaluation of the value and gradient of the following penalty function \cite[\S 7.2]{fletcher1973,orban}:
\begin{equation*}
\Phi_\sigma(x)=f(x)-g(x)^\top y_\sigma(x),
\end{equation*}
where $y_\sigma(x)\in\mathbb{R}^m$ is defined as the solution of the following minimization  problem:
\begin{equation*}
\min_y \|A(x)^\top y -\nabla f(x)\|^2+\sigma g(x)^\top y,
\end{equation*}
with $A(x)$ the Jacobian matrix of $g(x)$ at $x$ and $\sigma>0$,  a given real-valued penalty parameter. 

\subsection{Multilevel Levenberg-Marquardt method}

The second scenario, in which the solution of system \cref{base} is required, arises in the setting of multilevel Levenberg-Marquardt methods, which are specific members of the family of multilevel optimization methods recently introduced in Calandra et al \cite{multilevel_ann} and further analysed in Calandra et al \cite{high_order}.  The multilevel \LM method is intended to solve nonlinear least squares problems of the form: 
\begin{equation*}
\min_x f(x)=\frac{1}{2}\|F(x)\|^2,
\end{equation*}
with $F:\mathbb{R}^n\rightarrow\mathbb{R}^m$, a twice continuously differentiable 
function. In the two-level setting, the multilevel method allows to choose between two different models to compute the step at each iteration: the classical Taylor model 
or a cheaper model $m_k^H$, built starting from a given approximation $f^H(x^H)=\frac{1}{2}\|F^H(x^H)\|^2$ to the objective function:
\begin{equation*}
m_k^H(x_k^H,s^H)=\frac{1}{2} \|J^H(x_k^H)s^H+F^H(x_k^H)\|^2+\frac{\lambda_k}{2}\|s^H\|^2+( R \nabla f(x_k)-\nabla f^H(x_0^H))^\top s^H,
\end{equation*}
with $J^H(x_k^H)$ the Jacobian matrix of $F^H$ at $x_k^H$, $\lambda_k>0$ a real-valued regularization parameter, $R$ a full-rank linear restriction operator and $x_0^H=Rx_k$, $x_k$ denoting the current iterate at fine level. 

While minimizing the Taylor model amounts to solve a formulation based on the normal equations,  minimizing $m_k^H$ requires the solution of a problem of the form \cref{qls}, due to the addition of the correction term $( R \nabla f(x_k)-\nabla f^H(x_0^H))^\top s^H$, that is needed to ensure coherence between levels. Approximate minimization of the model is sufficient to guarantee convergence of the  method. Thus if the dimension of the coarse problem is not small, an iterative method is then particularly well suited for the minimization of the model.

\section{Conditioning of the problem and backward error analysis}\label{sec_teo}
In this section we first study the conditioning of the problem \cref{base}.  We first propose an explicit formula for the structured condition number, that is useful to compute a first order estimate of the forward error. We also propose theoretical results related to the backward error.

\subsection{Conditioning of the problem}\label{sec_cond}
We study the conditioning of problem \cref{base}, i.e. the sensitivity of the solution $x$ to perturbations in the data  $A, b, c$. We give an explicit formula for the structured condition number, for perturbations on all $A$, $b$ and $c$.  In the following, we define the \textit{condition number} of a function, see \cite{rice}.
\begin{definition}\label{def_cond}
	If $F$ is a continuously differentiable function 
	\begin{align*}
	F: \,&\mathcal{X}\rightarrow\mathcal{Y}\\
	x&\longmapsto F(x),
	\end{align*}
	the absolute condition number of $F$ at $x$ is the scalar $\|F'(x)\|_{\mathrm{op}}$.
	The relative condition number of $F$ at $x$ is 
	\begin{equation*}
	\frac{\|F'(x)\|_{\mathrm{op}}\,\|x\|_{\mathcal{X}}}{\|F(x)\|_{\mathcal{Y}}}.
	\end{equation*}
\end{definition}

\noindent We consider $F$ as the function that maps $A, b, c$ to the solution $x$ of \cref{base}:
\begin{align*}
F:\, & \mathbb{R}^{m\times n}\times \mathbb{R}^m\times \mathbb{R}^n\rightarrow\mathbb{R}^n\\
(A,b,c)\longmapsto& F(A, b, c)=A^\dagger b+A^\dagger(A^\dagger)^\top c.
\end{align*}
Let us define the true residual $r=b-Ax$. As in \cite{gratton1996condition}, we choose the Euclidean norm for the solution and the Frobenius norm (defined in the \textit{Notation} part of Section 1) for the data.

\vskip 5pt
In this section we mark perturbed quantities by a tilde.  Let the $m\times n$ matrix $A$ be perturbed to $\tilde{A}=A+\delta A$, the $m$-dimensional vector $b$ to $\tilde{b}=b+\delta b$, the $n$-dimensional vector $c$ to $\tilde{c}=c+\delta c$. $A^\top A$ is then perturbed to 
\begin{equation}\label{pert_ata}
\tilde{A}^\top\tilde{A}=(A+\delta A)^\top(A+\delta A)=A^\top A+\delta(A^\top A),\quad \delta(A^\top A)=A^\top\delta A+(\delta A)^\top A,
\end{equation}
neglecting the second order terms. 
The solution $x=(A^\top A)^{-1}(A^\top b+c)$ is then perturbed to $\tilde{x}=x+\delta x=(\tilde{A}^\top\tilde{A})^{-1}(\tilde{A}^\top\tilde{b}+\tilde{c})$.
Then, $\tilde{x}$ solves:
\begin{align*}
(A^\top A+\delta(A^\top A))\tilde{x}=(A^\top b+(\delta A)^\top b+A^\top\delta b+c+\delta c).
\end{align*}
Reminding \cref{pert_ata}, and that, since $A$ is full column rank, $A^\dagger=(A^\top A)^{-1}A^\top$ and  $(A^\top A)^{-1}=A^\dagger(A^\dagger)^\top$, we have
\begin{equation*}
\delta x=(A^\top A)^{-1}(\delta A)^\top r-A^\dagger\delta A x+A^\dagger\delta b+A^\dagger(A^\dagger)^\top\delta c.
\end{equation*}
From this, we conclude that 
\begin{equation*}
F'(A,b,c)(E,f,g)=(A^\top A)^{-1}E^\top r-A^\dagger E x+A^\dagger f+A^\dagger(A^\dagger)^\top g,
\end{equation*}
for all $E\in\mathbb{R}^{m\times n}, f\in\mathbb{R}^m,g\in\mathbb{R}^n$.
We then deduce the following property:
\begin{lemma} \label{lemma_cond}
	The conditioning of the problem \cref{base}, with Euclidean norm on the solution and  Frobenius norm on the data, is given by
	\begin{equation}
	\|F'(A,b,c)\|_{\mathrm{op}}=\|[(r^\top\otimes \ata )L_T+x^\top\otimes A^\dagger,A^\dagger,\ata]\|,
	\end{equation}
	where  $L_T$ is the linear operator such that $\mathrm{vec}(A^\top)=L_T\mathrm{vec}(A)$ ($L_T$ is a permutation matrix consisting of ones in the positions $(n(k-1)+l,m(l-1)+k)$ with $l=1,\dots,n$ and $k=1,\dots,m$ and of zeros elsewhere \cite{Fraysse2000,gratton1996condition}). 
\end{lemma}

In the following, we give an explicit and computable formula for the structured condition number. 

\begin{theorem}\label{teo_conditioning}
	The absolute condition number of problem \cref{base}, with Euclidean norm on the solution and  Frobenius norm on the data, is $\sqrt{\rule[0mm]{0mm}{3.15mm}\|\bar{M}\|}$, with $\bar{M} \in \mathbb{R}^{n \times n}$ given by
	\begin{equation}\label{conditioning}
	\bar{M}=(1+\|r\|^2)(A^\top A)^{-2}+(1+\|x\|^2)(A^\top A)^{-1}-2~\mathrm{sym}(B),
	\end{equation}
	with $B=A^\dagger rx^\top\ata$, $\mathrm{sym}(B)=\frac{1}{2}(B+B^\top)$ and $x$ the exact solution of \cref{base}. 
\end{theorem}
\begin{proof}
	Let us define $M=[(r^\top\otimes \ata )L_T+x^\top\otimes A^\dagger,A^\dagger,\ata]$.
	We remind that 
	\begin{equation}\label{fprime}
	\|F'(A,b,c)\|_{\mathrm{op}}=\|M\|=\|M^\top\|:=\sup_{y\neq 0}\frac{\|M^\top y\|}{\|y\|}.
	\end{equation}
	Let us consider 
	\begin{align*}
	y^\top F'(A,b,c)(E,f,g)&=y^\top(A^\top A)^{-1}E^\top r-y^\top A^\dagger Ex+y^\top A^\dagger f+y^\top\apa g\\
	&=r^\top E\ata y-y^\top A^\dagger Ex+y^\top A^\dagger f+y^\top\apa g.
	\end{align*}
	We remind that $E=\sum_{i=1}^{n}\sum_{j=1}^{m}e_i^\top Ee_j$, for $e_i,e_j$ the $i$-th and $j$-th vectors of the canonical basis. 
	Then, we can rewrite the expression as
	\begin{equation*}
	[ \mathrm{vec}(S),y^\top A^\dagger,y^\top\apa]\cdot[\mathrm{vec}(E),f,g]^\top:=w^\top[\mathrm{vec}(E),f,g]^\top,
	\end{equation*}
	introducing the matrix $S$ such that:
	\begin{equation*}
	S_{i,j}=r^\top e_ie_j^\top\ata y-y^\top A^\dagger e_ie_j^\top x.
	\end{equation*}
	It follows that $w=M^\top y$. We are then interested in the norm of $w$. We can compute the squared norm of $\mathrm{vec}(S)$ as:
	\begin{eqnarray*}
	\|\mathrm{vec}(S)\|^2 &=& \sum_{i=1}^{n}\sum_{j=1}^{m}(r^\top e_i e_j^\top\ata y-y^\top A^\dagger e_ie_j^\top x)^2\\
	&=&\|r\|^2\|\ata y\|^2+\|x\|^2\|(A^\dagger)^\top y\|^2\\
	& & -y^\top(A^\dagger rx^\top\ata +\ata xr^\top(A^\dagger)^\top) y.
	\end{eqnarray*}
	Then, 
	\begin{align*}
	\|w\|^2 =&~y^\top((1+\|r\|^2)(A^\top A)^{-2}+(1+\|x\|^2)A^\dagger(A^\dagger)^\top-(A^\dagger rx^\top\ata\\
	&+ \ata xr^\top(A^\dagger)^\top) y:=y^\top \bar{M}y.
	\end{align*}
	From \cref{fprime} 
	\begin{equation*}
	\|F'(A,b,c)\|_{\mathrm{op}}=\|M\|=\sup_{y\neq 0}\frac{\|M^\top y\|}{\|y\|}=\sup_{y\neq 0} \frac{\sqrt{y^\top \bar{M}y}}{\|y\|}=\sqrt{\rule[0mm]{0mm}{3.15mm}\|\bar{M}\|},
	\end{equation*}
	and we obtain the thesis.
\end{proof}

We remark here a fundamental difference with the least squares case. In the conditioning of least squares or structured conditioning of the normal equations, the term $\kappa^2(A)$ appears multiplied by the norm of the residual. This is interpreted saying that the sensitivity of the problem depends on $\kappa(A)$ for small or zero residual problems and on $\kappa^2(A)$ for all other problems \cite{bjorck1996numerical,higham2002accuracy}.  
In this case, we notice from \cref{conditioning} that a term proportional to $\kappa^2(A)$ will always be present.


\subsection{Backward error analysis}
In this section, we address the computation of the backward error, i.e. we consider the following problem.  Let $A\in\mathbb{R}^{m\times n}$, $b\in\mathbb{R}^m$, $c\in\mathbb{R}^n$ and $\tx$ a perturbed solution to \cref{base}.  Find the smallest perturbation $(E,f,g)$ of $(A,b,c)$ such that the vector $\tx$ exactly solves 
\begin{equation*}
(A+E)^\top(A+E)x=(A+E)^\top(b+f)+(c+g),
\end{equation*}
i.e. given
\begin{align*}
\mathcal{G}:= \{(E,f,g)&, E\in\mathbb{R}^{m\times n}, f\in\mathbb{R}^m, g\in\mathbb{R}^n\,:\nonumber\\
\,&(A+E)^\top(A+E)\tilde{x}=(A+E)^\top(b+f)+(c+g)\},
\end{align*}
we want to compute the quantity:
\begin{equation}\label{eta}
\eta(\tx,\theta_1,\theta_2)=\min_{(E,f,g)\in\mathcal{G}}\|[E,\theta_1 f,\theta_2 g]\|_F^2:=\min_{(E,f,g)\in\mathcal{G}}\|E\|_F^2+\theta_1^2 \|f\|^2+\theta_2^2 \|g\|^2,
\end{equation}
with $\theta_1,\theta_2$ positive parameters.
As usually done in the literature on backward errors, see, for example,  \cite{gratton1996condition,walden}, we measure the perturbations in a weighted Frobenius norm.

We prove specific results related to the backward error. We provide an explicit representation of the set of admissible perturbations on the matrix (\cref{teo_backward}) and a linearization estimate for $\eta(\tx,\theta_1,\theta_2)$ (\cref{lemma_linearization}), respectively.

\vskip 5pt

Given $v\in \mathbb{R}^m$, we define
\begin{equation*}
v^\dagger=\begin{cases}
\displaystyle \frac{v^\top}{\|v\|^2} \quad &\text{ if } v\neq 0,\\
0 \quad&\text{ if }	v=0.
\end{cases}
\end{equation*}
We remind the following properties that will be used later:
\begin{align*}
(I_m-vv^\dagger)v=0, && vv^\dagger v=v.
\end{align*}

We consider just perturbations on the matrix and we give next an explicit representation of the set of admissible perturbations.

\begin{theorem}\label{teo_backward}
	Let $A\in\mathbb{R}^{m\times n}$, $b\in\mathbb{R}^m$, $c,\tx\in\mathbb{R}^n$ and assume that $\tx\neq0$. Let $\tr=b-A\tx$ and define two sets $\mathcal{E,M}$ by
	\begin{align*}
	\mathcal{E}=&\{E\in\mathbb{R}^{m\times n}:\;(A+E)^\top(b-(A+E)\tx)=-c\,\},\\
	\mathcal{M}=&\{v\left(\alpha c^\top-v^\dagger A\right)+(I_m-vv^\dagger) (\tr\tx^\dagger+Z(I_n-\tx\tx^\dagger))\;:\\ &v\in\mathbb{R}^m,\,Z\in\mathbb{R}^{m\times n}, 
	\alpha \in \mathbb{R}, \, s.t.\,\, \alpha\|v\|^2(v^\dagger b-\alpha c^\top\tx)=-1\}.
	\end{align*}
	Then $\mathcal{E}=\mathcal{M}$. 
\end{theorem}

\begin{proof}
	The proof is inspired by that of \cite[Theorem 1.1]{walden}.  
	First, we prove $\mathcal{E}\subseteq\mathcal{M}$. Let $E\in\mathcal{E}$, then $E$ can be written as 
	\begin{equation}\label{decompE}
	E=(I_m-vv^\dagger)E\tx\tx^\dagger+vv^\dagger E+(I_m-vv^\dagger)E(I_n-\tx\tx^\dagger).
	\end{equation}
	We choose $v=\tr-E\tx$. Then $E\tx=\tr-v$ and
	\begin{equation}\label{eq2}
	(I_m-vv^\dagger)E\tx=(I_m-vv^\dagger)\tr.
	\end{equation} 
	Moreover, it holds 
	\begin{equation}\label{eq1}
	-c=(A+E)^\top (b-(A+E)\tx)=(A+E)^\top (\tr-E\tx)=(A+E)^\top v.
	\end{equation} 
	From \cref{eq1}, $v^\dagger E=-\displaystyle \frac{c^\top}{\|v\|^2}-v^\dagger A$. 
	Hence, from relations \cref{decompE}, \cref{eq2} and \cref{eq1}
	\begin{equation*}
	E=(I_m-vv^\dagger)\tr\tx^\dagger-v\left(\frac{c^\top}{\|v\|^2}+v^\dagger A\right)+(I_m-vv^\dagger)E(I_n-\tx\tx^\dagger).
	\end{equation*}
	Then $E\in\mathcal{M}$ with $v=\tr-E\tx$, $\alpha=-\displaystyle \frac{1}{\|v\|^2}$ and $Z=E$, as 
	\begin{align*}
	\alpha \|v\|^2(v^\dagger b-\alpha c^\top\tx)&=-\frac{1}{\|v\|^2}(v^\top b+ c^\top\tx)=-\frac{1}{\|v\|^2}(v^\top b-v^\top(A+E)\tx)\\
	&=-\frac{1}{\|v\|^2}v^\top(\tr-E\tx)=-1,
	\end{align*}
	where the second equality follows from \cref{eq1}.
	
	Conversely, let $E\in\mathcal{M}$. Then,
	\begin{align}
	E\tx&=\alpha vc^\top\tx -vv^\dagger A\tx+\tr-vv^\dagger\tr=	\alpha vc^\top\tx -vv^\dagger b+\tr,\label{eq11}\\
	E^\top v&=\alpha\|v\|^2c-A^\top v\label{eq12}.  
	\end{align}
	Then,
	\begin{align*}
	(A+E)^\top(b-(A+E)\tx)&=(A+E)^\top(\tr-E\tx)=(A+E)^\top(vv^\dagger b-\alpha vc^\top\tx)\\&=(A+E)^\top v(v^\dagger b-\alpha c^\top\tx)=
	\alpha\|v\|^2c(v^\dagger b-  \alpha c^\top\tx)=-c,
	\end{align*}
	where the second equality follows from \cref{eq11}, the fourth one from \cref{eq12} and the last one from the constraint in $\mathcal{M}$.
	From this we conclude that $E\in\mathcal{E}$, which completes the proof. 
\end{proof}
Let us remark that if $c=0$ we recover the known result for least squares problems given in \cite[Theorem 1.1]{walden}. We also note that the parametrization of the set of perturbations $\mathcal{E}$ is similar to that obtained for equality constrained least squares problems in \cite{cox1999}, where the constraint is however not present.   

Due to the constraint in $\mathcal{M}$, it is rather difficult to find an analytical formula for $\eta(\tx,\theta_1,\theta_2)$. It is however easy to find a linearization estimate for $\eta(\tx,\theta_1,\theta_2)$ with $\theta_1,\theta_2$ stricly positive, i.e.,  given 
\begin{equation*}
h(A,b,c,x)=A^\top(b-Ax)+c,
\end{equation*} 
to find $(E,f,g)$ such that
\begin{subequations}\label{lin}
	\begin{align}
	\bar{\eta}(\tx,\theta_1,\theta_2)&=\min\|[E,\theta_1 f,\theta_2 g]\|_F\quad\text{s.t.}\\
	h(A,b,c,\tx)&+[J_A,\theta_1^{-1}J_b,\theta_2^{-1}J_c]\begin{bmatrix}
	\mathrm{vec}(E)\\\theta_1 f\\\theta_2 g
	\end{bmatrix}=0,
	\end{align}
\end{subequations} 
where $J_A$, $J_b$ and $J_c$ are the Jacobian matrices of $h$ with respect to $\mathrm{vec}(A)$, $b$, $c$ \cite{chang}. 

\begin{lemma}\label{lemma_linearization}
	Let ${\eta}(\tx,\theta_1,\theta_2)$ be defined as in \cref{eta}, $\bar{\eta}(\tx,\theta_1,\theta_2)$ be defined as in \cref{lin} and $\tilde{r} = b - A \tilde{x}$.
	Then \begin{equation*}
	\bar{\eta}(\tx,\theta_1,\theta_2)=\Bigg\|\begin{bmatrix}
	vec(E)\\\theta_1 f\\
	\theta_2 g
	\end{bmatrix}\Bigg\|=\|J^\dagger h(A,b,c,\tx)\|,
	\end{equation*}
	with 
	\begin{align*}
	J:=[I_n\otimes \tilde{r}^\top-A^\top(\tx\otimes I_m),\theta_1^{-1}A^\top,\theta_2^{-1}I_n].
	\end{align*}
	Moreover, assume that $\tilde{r}\neq 0$. If $4\eta_1\|J^\dagger\|\eta(\tx,\theta_1,\theta_2)\leq1$, then
	\begin{equation*}
	\frac{2}{1+\sqrt{2}}~\bar{\eta}(\tx,\theta_1,\theta_2)\leq\eta(\tx,\theta_1,\theta_2)\leq 2~\bar{\eta}(\tx,\theta_1,\theta_2),
	\end{equation*}
	where $\eta_1=\sqrt{\theta_1^{-2}+\theta_2^{-2}+\|\tx\|^2}$.
\end{lemma}

\begin{proof}
	The first assertion follows as a simple modification of the results in Section 2 of \cite{liu}. Simply adding the term corresponding to $c$ in the linearization leads to:
	\begin{align*}
	[J_A,\theta_1^{-1}J_b,\theta_2^{-1}J_c]=[I_n\otimes \tilde{r}^\top-A^\top(\tx\otimes I_m),\theta_1^{-1}A^\top,\theta_2^{-1}I_n]=J,
	\end{align*}
	from which the result follows. The second result can be obtained repeating the arguments of \cite[Corollary 2]{liu}.
\end{proof}

The linearized estimate is usually called an asymptotic estimate, as it becomes exact in the limit for $\tilde{x}$ that tends to the exact solution $x$ \cite{liu}. It also has the advantage of being easily computable.

\vskip 5pt
In this section we have considered theoretical questions related to the solution of \cref{base}. In the following, we will consider the approximation of a solution in finite arithmetic.


\section{Finite arithmetic implementation of CG for {$A^\top A\lowercase{x}=A^\top\lowercase{b}+\lowercase{c}$}}\label{sec_error_analysis}
In this section, we address the implementation of CG in finite arithmetic. We first discuss the well-known case $c=0$, that corresponds to the normal equations for least squares problems.  In particular, we briefly remind the issues arising in finite precision when applying CG to such systems, and the successful algorithmic ingredients that make CGLS more stable than CG. We then consider the more general case $c\neq 0$ and show that the same issues also arise in this case. These observations are found crucial for the development of stable solution methods for \cref{base}.  In  \cref{sec_methods} indeed, the successful ingredients employed in the case $c=0$ will be extended to the more general case $c\neq 0$. 

\subsection{Case $c=0$}
In \cite{bjorck1998stability} the authors study the implementation of CG applied to the normal equations in finite arithmetic and, thanks to a rounding error analysis, individuate the operations which may lead to amplifications of initial errors in the data, and so to large perturbations in the solution. Then, they propose the CGLS method, which is a more stable implementation of CG for normal equations. For convenience, we report both algorithms, in \cref{algo_cg} and \cref{algo_cgls}, respectively.

Two different sources of numerical error do occur when applying CG on the normal equations. First, computing $\mathrm{fl}(A^\top A)$ can lead to a loss of information.  CGLS avoids forming matrix $A^\top A$ explicitly. Matrix-vector products are thus computed exploiting subsequently the action of $A$ and of its transpose on a vector, whereas the computation of $p_k^\top A^\top Ap_k$ is performed as $\|Ap_k\|^2$.  The second source of error comes from the right-hand side $b$. Assuming that $x_0 = 0$, the only information related to  $b$ is in the initialization of $p_1=A^\top (b-Ax_0)$, see \cref{algo_cg}. No further reference to $b$ is made in the iterative procedure. It then follows that roundoff errors that occur in computing the vector $A^\top b$ cannot be compensated for later \cite{bjorck1998stability}.
Hence, in CGLS, the residual of the normal equations $r_k=A^\top(b-Ax_k)$ is not recurred directly. Rather, $d_k=b-Ax_k$ is recurred, while $r_k$ is then updated multiplying $d_k$ by $A^\top$. The next result is also proved, which is an extension of the corresponding result in \cite{greenbaum}, where the author studies the finite precision implementation of the class of iterative methods considered in  \cref{lemma_bjork}. 

\begin{lemma}\label{lemma_bjork}Let $A\in\mathbb{R}^{m\times n}$. 
	Consider an iterative method in which each step updates the approximate solution $x_k$ and the residual $r_k$ of the system $Ax=b$ using the two formulas:
	\begin{align*}
	x_{k+1}&=x_k+\alpha_k~p_k,\\
	r_{k+1}&=r_k-\alpha_k~Ap_k,\\
	\end{align*}
	with $\alpha_k \in \mathbb{R}$, $p_k \in \mathbb{R}^n$ and $x \in \mathbb{R}^n$ the true solution. 
	The difference between the true residual $b-Ax_k$ and the recursively computed residual  $r_k$ satisfies
	\begin{equation}\label{rel_res}
	\frac{\|b-Ax_k-r_k\|}{\|A\|\|x\|}\leq u~O(k)\left(1+\Theta_k+\frac{\|r\|}{\|A\|\|x\|}\right),
	\end{equation}
	with $u$ the machine precision, $r=b-Ax$ and $\Theta_k=\max_{j\leq k}\|x_j\|/\|x\|$.
\end{lemma}

 \cref{lemma_bjork} is used in \cite{bjorck1998stability} to deduce a bound on the forward error for such methods: 
\begin{equation}\label{bound_error}
\frac{\|x-x_k\|}{\|x\|}\leq \kappa(A)~u~O(k)\left(3+\frac{\|r\|}{\|A\|\|x\|}\right)+\kappa(A)\frac{\|r-r_k\|}{\|A\|\|x\|}.
\end{equation}
It is shown in \cite{bjorck1998stability,greenbaum} that the residuals $r_k$ generated by CG tend towards zero, so that $x_k$ will eventually remain unchanged. In the bound $k$ can therefore be replaced by $S$, the number of iterations necessary to reach this steady-state. 
Then, if 
\begin{equation}\label{cond_bjork}
\frac{\|r-r_S\|}{\|A\|\|x\|}\leq c_1 u
\end{equation}
for a strictly positive constant $c_1$, we deduce that the method may compute more accurate solutions than a backward stable method \cite{bjorck1998stability}. 

\subsection{Case $c\neq 0$}
When using CG applied to \cref{base}, observations similar to those made in the previous section can be made. If the right-hand side $A^\top b+c$ is formed, an error occurs. We can indeed quantify this error as in \cite{bjorck1998stability}, \S 4. Following \cite{bjorck1998stability}, we assume the following model of floating point arithmetic on a machine with roundoff $u$:
\begin{equation*}
\mathrm{fl}(a\, \mathrm{op}\, b)=(a \,\mathrm{op}\, b)(1+\delta(ab)), \quad \lvert \delta (ab)\rvert\leq u, \quad \mathrm{op}\in\{+,-,\times,\div\}
\end{equation*}
This leads to a bound for the roundoff error in the product of two matrices $A\in\mathbb{R}^{m\times n}$ and $B\in \mathbb{R}^{n\times p}$:
\begin{equation}\label{bound_prodm}
\lvert \mathrm{fl}(AB)-AB\rvert\leq \gamma_n\lvert A\rvert\lvert B\rvert,\qquad \gamma_n=nu/(1-nu),
\end{equation}
where the inequalities are to be interpreted componentwise \cite{bjorck1998stability,higham2002accuracy}. It is assumed that $nu \ll 1$.

Errors are introduced when the product $A^\top b$ and the sum with $c$ in the right-hand side are performed. The perturbed solution $\tx=x+\delta x$ corresponding to the system with perturbed right-hand side $\mathrm{fl}(\mathrm{fl}(A^\top b)+c)$ will satisfy, assuming that the linear systems are solved exactly:
\begin{equation}\label{system_pert}
A^\top A\tx=\mathrm{fl}(\mathrm{fl}(A^\top b)+c).
\end{equation}
From relation \cref{bound_prodm} we obtain
\begin{align}\label{rel}
\mathrm{fl}(A^\top b)&=A^\top b+\delta_1,\quad\quad \lvert \delta_1\rvert \leq \gamma_m\lvert A^\top\rvert \lvert b\rvert,\\
\mathrm{fl}(\mathrm{fl}(A^\top b)+c)&=A^\top b+c+\delta_1+\delta_2,\quad \lvert\delta_2\rvert\leq u \lvert A^\top b+c\rvert+O(u^2)\label{rel2}.
\end{align}
Subtracting \cref{base} from \cref{system_pert} and solving for $\delta x$, we obtain from \cref{rel}-\cref{rel2}:
\begin{align}
\lvert \delta x\rvert \leq& \lvert (A^\top A)^{-1}\rvert \lvert\delta _1+\delta_2\rvert\leq \gamma_{m}\lvert (A^\top A)^{-1}\rvert\lvert A^\top\rvert\lvert b\rvert+u \lvert (A^\top A)^{-1}\rvert \lvert A^\top b+c\rvert\nonumber\\
\leq & \gamma_{m+1}\lvert (A^\top A)^{-1}\rvert\lvert A^\top\rvert\lvert b\rvert+u \lvert (A^\top A)^{-1}\rvert \lvert c\rvert
 \label{et}
\end{align}
From the singular value decomposition of $A$, it is possible to deduce the following relations, see \cite{bjorck1996numerical} \S1.4.3: 
\begin{equation*}
\|(A^\top A)^{-1}\|=\frac{1}{\sigma_{\min(A)}^2},\, \|A\|=\sigma_{\max}(A), \, \|A^\dagger\|=\frac{1}{\sigma_{\min(A)}}.
\end{equation*}
Exploiting these relations and extending inequality \cref{et} to the norm, we obtain:
\begin{align}
\|\delta x\|\leq& \gamma_{m+1}\| (A^\top A)^{-1}\| \|A^\top\|\| b\|+u \| (A^\top A)^{-1}\|\| c\|, \nonumber\\
\le &\gamma_{m+1}\kappa(A) \|A^\dagger\|\| b\|+u \| A^\dagger\|^2\| c\|, \nonumber\\
\le &u \kappa^2(A) \left(\frac{m+1}{1-(m+1)u}\frac{\|b\|}{\|A\|}+\frac{\|c\|}{\|A\|^2}\right)\label{bound}.
\end{align}
Since $b$ and $c$ are not used in the iterative process, this initial error will not be cancelled and the best forward error bound we can hope for will include the term given above, which may be larger than the error introduced by the solution of the linear system, as we will see in  \cref{sec_num}.
Then, if this initial rounding error is larger than the optimal bound
$ \sqrt{\rule[0mm]{0mm}{3.15mm}\|\bar{M}\|}\|[A,b,c]\|_Fu$ we can expect CG to produce less than optimal solutions.
This will happen for example if    
\begin{equation*}
	\|b\|\|A\|+\|c\|>>
	\left[1+\|r\|+2\sqrt{\|c\|\|x\|}+\frac{1+\|x\|}{\|A^\dagger\|}\right]\sqrt{\|A\|_F^2+\|b\|^2+\|c\|^2},
	\end{equation*} 
	where the second term has been obtained from an upper bound to the condition number $ \sqrt{\rule[0mm]{0mm}{3.15mm}\|\bar{M}\|}$.
	
 Similarly, the residual of the perturbed solution may differ from the true residual by as much as
\begin{align*}
&\|A^\top b+c-(A^\top A)(x + \delta x) -(A^\top b+c-A^\top A x)\| =\\
&=\|A^\top A\delta x\|\leq \gamma_{m+1} \kappa(A)^2 \|A\|\| b\|+u \kappa(A)^2\| c\|,
\end{align*}
and again, this cannot be cancelled in the iterations.

\begin{algorithm}
	\caption{Conjugate Gradient for $Ax=b$ (CG)}
	\label{algo_cg}
	\begin{algorithmic}
		\STATE{Input: $A$, $b$, $x_0$.}
		\STATE{Define 	$r_0=b-Ax_0$, $p_1=r_0$.}\\
		\FOR{ $k=1,2,\dots$ }
		\STATE{$\alpha_k=\displaystyle \frac{r_{k-1}^\top r_{k-1}}{p_k^\top Ap_k}$,}
			\STATE{$x_{k}=x_{k-1}+\alpha_kp_k$,}
		\STATE{$r_{k}=r_{k-1}-\alpha_kAp_k$, }
		\STATE{$\beta_{k}=\displaystyle \frac{r_{k}^\top r_{k}}{r_{k-1}^\top r_{k-1}}$,}
		\STATE{$p_{k+1}=r_{k}+\beta_{k}p_k$.}
		\ENDFOR
	\end{algorithmic}
\end{algorithm}

\begin{algorithm}
	\caption{Conjugate Gradient Least Squares for $A^\top Ax=A^\top b$ (CGLS)}
	\label{algo_cgls}
	\begin{algorithmic}
		\STATE{ Input: $A$, $b$, $x_0$.}
		\STATE{Define 	$d_0=b-Ax_0$, $r_0=A^\top d_0$, $p_1=r_0$.}\\
		\FOR{ $k=1,2,\dots$ }
		\STATE{$t_{k}=Ap_{k}$,}
		\STATE{ $\alpha_k=\displaystyle \frac{r_{k-1}^\top r_{k-1}}{t_k^\top t_k}$,}
		\STATE{ $x_{k}=x_{k-1}+\alpha_kp_k$,}
		\STATE{$d_{k}=d_{k-1}-\alpha_kt_k$,}
		\STATE{$r_{k}=A^\top d_{k}$,}
		\STATE{$\beta_{k}=\displaystyle \frac{r_{k}^\top r_{k}}{r_{k-1}^\top r_{k-1}}$,}
		\STATE{$p_{k+1}=r_{k}+\beta_{k}p_k$.}
		\ENDFOR
	\end{algorithmic}
\end{algorithm}

Thus, when solving system \cref{base}, we encounter the same problems as with the normal equations. Our aim is then to exploit the particular structure of \cref{base}, that is close to that of normal equations, to extend the successful algorithmic ingredients used in that case to design stable methods in our specific setting. 

\section{Proposed iterative solution methods}\label{sec_methods}
In this section, we introduce two solution methods for \cref{base}. They exploit specific reformulations of the original problem or approximations to it, to be able to factorize $A$ on both the left- and the right-hand sides. 

Having in mind the possibility of seeking an approximate solution, we especially focus on iterative methods, that are intended to be more stable than the standard CG method, later used as a reference method. We propose and describe the two methods, named CGLS$\epsilon$ and CGLS$I$, respectively. We first explain the ideas on which their designs are based and then present an extensive numerical evaluation in  \cref{sec_num}.

\subsection{CGLS$\epsilon$}

The first approach we propose is based on the solution of a least squares problem, with both matrix and right-hand side depending on an auxiliary real-valued strictly positive parameter $\epsilon$. The problem is defined in a way that its solution  $x_\epsilon$ approximates the true solution of \cref{base} $x$, when the parameter $\epsilon$ goes to zero. This reformulation is especially convenient because we can exploit the structure of the normal equations and solve the problem with CGLS, hence avoiding the issues described in  \cref{sec_error_analysis}.

Given $\epsilon>0$, let us then define
\begin{equation}\label{def_eps}
A_\epsilon=\begin{bmatrix}
A\\\epsilon~c^\top
\end{bmatrix},\quad b_\epsilon=\begin{bmatrix}
b\\\displaystyle 1/\epsilon
\end{bmatrix}.
\end{equation}
We then consider the following linear least squares problem:
\begin{equation}\label{sistema_eps}
\min_{x}\|A_\epsilon x-b_\epsilon\|^2.
\end{equation}
We remark that, due to the presence of $\epsilon$ in the right-hand side which is expected to be a small quantity, the residual will generally be really large. From standard perturbation theory on least squares problems, when the residual is large, the sensitivity of the system depends on $\kappa^2(A_\epsilon)\|A_\epsilon x_\epsilon-b_\epsilon\|/(\|A_\epsilon\|\|x_\epsilon\|)$. Hence, problem \cref{sistema_eps} would then generally be severely ill-conditioned.  However, the classical results are not so meaningful in this case, as the components of $b_\epsilon$ and $\Ae$ have different magnitudes, see Remark 1.4.3 in Bj\"orck \cite{bjorck1996numerical}.

A more meaningful analysis is proposed next in  \cref{teo_conditioning_eps}, whose proof can be easily derived analogously to that of \cref{teo_conditioning}. 

\begin{theorem}\label{teo_conditioning_eps}
	Let $F_{\epsilon}$ be the function that maps $A, b, c$ to the solution $x_\epsilon$ of \cref{sistema_eps}:
	\begin{align*}
	F_\epsilon: \, & \mathbb{R}^{m\times n}\times \mathbb{R}^m\times \mathbb{R}^n\rightarrow\mathbb{R}^n\\
	(A,b,c)\longmapsto& F_\epsilon(A,b,c)=(\Ae^\top\Ae)^{-1}(A^\top b+c),
	\end{align*}
	and let $r_\epsilon=b-Ax_\epsilon$.
	The absolute condition number of problem \cref{sistema_eps}, with Euclidean norm on the solution and Frobenius norm on the data, is then given by:
	\begin{align*}
	\|F'_\epsilon(A,b,c)\|_{\mathrm{op}}=\|&[(r_\epsilon^\top\otimes (\Ae^\top\Ae)^{-1} )L_T+x_\epsilon^\top\otimes (\Ae^\top\Ae)^{-1}A^\top,\\&(\Ae^\top\Ae)^{-1}A^\top,(1-2\epsilon c^\top x_\epsilon)(\Ae^\top\Ae)^{-1}]\|,
	\end{align*}
	where $L_T$ is the linear operator introduced in  \cref{lemma_cond}. 
	This condition number can be equivalently expressed as $\sqrt{\rule[0mm]{0mm}{3.15mm}\|\bar{M}_\epsilon\|}$, with
	\begin{align}\label{conditioning_eps}
	\bar{M}_\epsilon=&~((1-2\epsilon c^\top x_\epsilon)^2+\|r_\epsilon\|^2)(\Ae^\top\Ae)^{-2}+(1+\|x_\epsilon\|^2)(\Ae^\top\Ae)^{-1}A^\top A(\Ae^\top\Ae)^{-1}\nonumber\\&-2~\mathrm{sym}(B_\epsilon)
	\end{align}
	with $B_\epsilon=(\Ae^\top\Ae)^{-1}A^\top r_\epsilon x_\epsilon^\top(\Ae^\top\Ae)^{-1}$ and $\mathrm{sym}(B_\epsilon)=\frac{1}{2}(B_\epsilon+B_\epsilon^\top)$.
\end{theorem}

It is essential to notice that from \cref{teo_conditioning_eps}, the conditioning of \cref{sistema_eps} does not depend on $\|b_\epsilon-A_\epsilon x_\epsilon\|$, that will be really large, but rather on $\|r_\epsilon\|=\|b-Ax_\epsilon\|$, that will be indeed much smaller. The normal equations for \cref{sistema_eps} are 
\begin{equation}\label{syst_eps}
(A^\top A+\epsilon^2 cc^\top)x=A^\top b+c,
\end{equation}
that is, if $\epsilon$ goes to zero, the matrix of \cref{syst_eps} is close to that of the normal equations \cref{base}. We can then prove that the solution of \cref{syst_eps} approximates a solution of \cref{base} for small $\epsilon$.
\begin{lemma}\label{lemma_eps}
	Let $\xe$ be the solution of \cref{syst_eps} and $x$ be the solution of \cref{base}. Then, $\displaystyle{\lim_{\epsilon\rightarrow 0}\xe=x}$ and the relative norm of the error satisfies
	\begin{equation}\label{error_estim_eps}
	\frac{\|\xe-x\|}{\|x\|}\leq\epsilon^2\frac{\|c\|\|w\|}{1+\epsilon^2c^\top w}, \quad  w=(A^\top A)^{-1}c.
	\end{equation} 
\end{lemma}
\begin{proof}
	$A^\top A+\epsilon^2cc^\top$ is obtained as a rank one correction of $A^\top A$. Thus the application of the Sherman-Morrison formula \cite[\S 3.1.2]{bjorck1996numerical} leads to:
	\begin{equation*}
	\xe=x-\frac{\epsilon^{2}c^\top x}{1+\epsilon^{2}c^\top w}w,\quad  w=(A^\top A)^{-1}c.
	\end{equation*}
	Then we obtain
	\begin{equation*}
	\frac{\|\xe-x\|}{\|x\|}\leq\epsilon^2\frac{\|c\|\|w\|}{1+\epsilon^2c^\top w},
	\end{equation*}
	which completes the proof.
\end{proof}

This result is helpful to propose a practical choice for $\epsilon$, that can be related to $\kappa(A)$. Note however that, to have guarantee of obtaining an accurate solution, a relatively small $\epsilon$ may be needed, depending on $\|w\|$. One may be then concerned that a really small $\epsilon$ may cause large errors in finite arithmetic, but we can show that an arbitrarily small value can indeed be used. A perturbed solution $\tilde{x}_\epsilon=x_\epsilon+\delta x_\epsilon$ will be such that:
\begin{equation*}
(\Ae^\top\Ae)(\delta x_\epsilon)=\delta(\Ae^\top\be).
\end{equation*} 
We need little care in evaluating $\delta(\Ae^\top\be)$, as the standard bound
\begin{equation*}
\lvert\delta(\Ae^\top\be)\rvert\leq\gamma_{m+1}\lvert\Ae^\top\lvert\rvert\be\rvert
\end{equation*}  
will largely overestimate it, due to the possibly large norm of $b_\epsilon$. Indeed, in practical computations, the large last component of $b_\epsilon$ is compensated by the multiplication with the small entries in the last line of $A_\epsilon$. We consider that
\begin{equation*}
\mathrm{fl}(\Ae^\top\be)=\mathrm{fl}(A^\top b)+\mathrm{fl}\left(\epsilon c\frac{1}{\epsilon}\right)+\delta_s,
\end{equation*}  with $\delta_s$ an additional error due to the summation. 
We remark that if we assume $\epsilon$ to represent a machine number, i.e. $\epsilon=2^i$ for $i\in\mathbb{Z}$, then $\mathrm{fl}\left(\epsilon~c~\displaystyle \frac{1}{\epsilon}\right)=c$, the computation is performed exactly, as it just amounts of a shift in the exponent. Then,
\begin{equation*}
\mathrm{fl}\left(\Ae^\top\be\right)=A^\top b+c+\delta_p+\delta_s, \;\; \text{ with }\;\; \lvert\delta_s\rvert\leq u\lvert A^\top b+c\rvert+O(u^2),\; \lvert\delta_p\rvert\leq\gamma_m \lvert A\rvert\lvert b\rvert,
\end{equation*} 
and the bound does not depend on $\epsilon$.
\vskip5pt
The first approach we propose, called CGLS$\epsilon$, is based on the solution of \cref{sistema_eps} with the CGLS method. From  \cref{lemma_eps}, if we choose a small value for $\epsilon$, we can expect the computed solution to be an accurate approximation to $x$. We will later discuss the choice of this free parameter in  \cref{sec_num_eps}. 

\subsection{CGLS$I$}
The second method we propose is still based on the solution of an augmented system. In this case, we do not employ any auxiliary parameter.  We rather use an auxiliary matrix $\hat{I} \in \mathbb{R}^{(m+1) \times (m+1)}$ and we define $\hat{A} \in \mathbb{R}^{(m+1) \times n}$ and $\hat{b} \in \mathbb{R}^{m+1}$, accordingly:
\begin{equation*}
\hat{A}=\begin{bmatrix}
A\\c^\top
\end{bmatrix}, \qquad \hat{I}= \left \lbrack \begin{array}{cc}
I_m & 0 \\
0   & 0 \\
\end{array} \right \rbrack ,\qquad \hat{b}=\begin{bmatrix}
b\\1
\end{bmatrix}.
\end{equation*}
We then reformulate \cref{base} as:
\begin{equation*}
\hat{A}^\top\hat{I}\hat{A}x=\hat{A}^\top\hat{b},
\end{equation*}
the two systems being equivalent. Although this reformulation is not a least squares problem, it still offers the advantage to factorize $\hat{A}^\top$ in both the right- and the left-hand sides. Hence, we can use a Krylov subspace method as a solution method, that can be implemented with no need of forming matrix $\hat{A}^\top\hat{I}\hat{A}$ explicitly or of recurring the residual $r=\hat{A}^\top(\hat{I}\hat{A}x-\hat{b})$. We can rather simply update $\hat{d}=\hat{I}\hat{A}x-\hat{b}$ along the iterations and form $r$ by multiplication with $\hat{A}^\top$. This allows to avoid corruption of $b$ by multiplication with $A^\top$ at the beginning of the process, without possibility of amending for it. We describe the whole procedure in  \cref{algo_i}. 

\begin{algorithm}
	\caption{CGLS$I$ for $A^\top Ax=A^\top b+c$}
	\label{algo_i}
	\begin{algorithmic}
	\STATE{ Input: $\hat{A}$, $\hat{b}$, $x_0$}
	\STATE{Define  $\hat{d}_0=\hat{b}-\hat{A}x_0$, $r_0=\hat{A}^\top(\hat{b}-\hat{A}x_0)$, $p_1=r_0$.}
	\FOR{$k=1,2,\dots$}
		\STATE{ $\hat{t}_{k}=\hat{I}\hat{A}p_{k}$},
		\STATE{ $\alpha_k=\displaystyle \frac{r_{k-1}^\top r_{k-1}}{\hat{t}_k^\top \hat{t}_k}$},
		\STATE{$x_{k}=x_{k-1}+\alpha_kp_k$},
		\STATE{$\hat{d}_{k}=\hat{d}_{k-1}-\alpha_k \hat{t}_k$},
		\STATE{$r_{k}=\hat{A}^\top \hat{d}_{k}$},
		\STATE{$\beta_{k}=\displaystyle \frac{r_{k}^\top r_{k}}{r_{k-1}^\top r_{k-1}}$},
		\STATE{$p_{k+1}=r_{k}+\beta_{k}p_k$}.
		\ENDFOR
		\end{algorithmic}
\end{algorithm}

In exact arithmetic the method in  \cref{algo_i} is equivalent to CG applied to \cref{base}.  The practical implementation of CGLS$I$ just relies on the use of augmented matrices together with two expedients to make the method more stable: the computation of  $p_k^\top A^\top A p_k$ as $\|\hat{I}\hat{A}p_k\|^2$ and the update of the residual. We can therefore expect the same benefits of CGLS as compared to CG.

\vskip 5pt

\begin{remark}\label{remark_bjork}
	 \cref{lemma_bjork} also applies to the CGLS$I$ method with $A$ replaced with $\hat{I}\hat{A}$. We can then deduce that, if condition \cref{cond_bjork} holds, the bound \cref{bound_error} is also valid and CGLS$I$ will provide numerical solutions to \cref{base} in a stable way. As pointed out in \cite{bjorck1998stability}, proving condition \cref{cond_bjork} is not an easy task. Nevertheless, we have found that this condition is shown to be satisfied numerically after extensive numerical experimentations, see  \cref{sec_num}. We also point out that the condition number of $\hat{A}$ can be related to that of $A$ by \cite[Corollary 7.3.6]{horn2012matrix}. We finally notice that \cref{lemma_bjork} also applies to CGLS$\epsilon$. Nevertheless it is not so meaningful, since in \cref{rel_res} the quantity $\|A_\epsilon x_\epsilon-b_\epsilon\|$ would be rather large. 
\end{remark}


\section{First order approximation for the forward error}\label{sec_forward}
The formulas we have derived in  \cref{teo_conditioning} and  \cref{teo_conditioning_eps} for the structured condition number of problems \cref{base} and \cref{sistema_eps} can be used to  provide a first order estimate of the forward error on a solution obtained by a method that does not form the matrix $A^\top A$ or $A_\epsilon^\top A_\epsilon$ explicitly. We define these estimates as the product of the relative condition number (see \cref{def_cond}) and the linearized estimate of the backward error in \cref{lemma_linearization}.

For CGLS$I$, we use the relative condition number
\begin{equation}\label{rel_cond}
\frac{\sqrt{\rule[0mm]{0mm}{3.15mm}\|\bar{M}\|}\|[A,b,c]\|_F}{\|x\|}
\end{equation} 
of \cref{base} for perturbations on $A$, $b$, $c$, with $\bar{M}$ defined in \cref{conditioning}. 

For CGLS$\epsilon$, the error on the computed solution $\hat{x}_\epsilon$ depends on two terms, the distance of the computed solution to the exact solution of \cref{syst_eps} and the distance of the exact solution of \cref{syst_eps} to the exact solution of \cref{base}:
\begin{equation}\label{err_eps}
\frac{\|x-\hat{x}_\epsilon\|}{\|x\|}\leq \frac{\|x-x_\epsilon\|}{\|x\|}+\frac{\|x_\epsilon-\hat{x}_\epsilon\|}{\|x\|}=
\frac{\|x-x_\epsilon\|}{\|x\|}+\frac{\|x_\epsilon-\hat{x}_\epsilon\|}{\|x_\epsilon\|}\frac{\|x_\epsilon\|}{\|x\|}.
\end{equation}

The first term and the ratio $\displaystyle \frac{\|x_\epsilon\|}{\|x\|}$ can be bounded  by the Sherman-Morrison formula, see \cref{error_estim_eps}, while the second one by using the relative condition number
\begin{equation}\label{rel_cond_eps}
\frac{\sqrt{\rule[0mm]{0mm}{3.15mm}\|\bar{M}_\epsilon\|}\|[A,b,c]\|_F}{\|x\|}
\end{equation} 
obtained in \cref{teo_conditioning_eps}, with $\bar{M}_\epsilon$ defined  in \cref{conditioning_eps}.

For CG, we can use the same bound as for CGLS$I$, adding the initial rounding error \cref{bound} due to the computation of the right-hand side.

\vskip 5pt

Given the computed solution $\hat{x}$, we then define the following error estimates: 
\begin{subequations}\label{cond}
	\begin{align}
	& \hat{E}_{CGLSI}:=\frac{\sqrt{\rule[0mm]{0mm}{3.15mm}\|\bar{M}\|}\|[A,b,c]\|_F}{\|\hat{x}\|}\bar{\eta}(\hat{x}),\\
	& \hat{E}_{CGLS_\epsilon}:= \epsilon^2\frac{\|c\|\|w\|}{1+\epsilon^2c^\top w}+\frac{\sqrt{\rule[0mm]{0mm}{3.15mm}\|\bar{M}_\epsilon\|}\|[A,b,c]\|_F}{\|\hat{x}\|}\bar{\eta}_\epsilon(\hat{x}_\epsilon) \Big\|I_n-\frac{\epsilon^2wc^\top}{1+\epsilon^2c^\top w}\Big\|,\\
	& \hat{E}_{CG}:=\frac{\sqrt{\rule[0mm]{0mm}{3.15mm}\|\bar{M}\|}\|[A,b,c]\|_F}{\|\hat{x}\|}\bar{\eta}(\hat{x})+ \kappa^2(A)\bar{\eta}(\hat{x}) \left(\frac{m+1}{1-(m+1)u}\frac{\|b\|}{\|A\|}+\frac{\|c\|}{\|A\|^2}\right),\\
	\end{align}
\end{subequations}
$\bar{\eta}(\hat{x}),\bar{\eta}_\epsilon(\hat{x}_\epsilon)$ being the linearized estimates of the backward error of the corresponding problem.

We will show in  \cref{sec_num_bounds} that the proposed estimates nicely predict the actual error in the numerical simulations and that in all cases they provide upper bounds for the forward error. 

\section{Numerical experiments}\label{sec_num}

We numerically validate the performance of the methods presented in  
\cref{sec_methods}. We discuss their practical implementation and evaluate the first order estimates of the forward error based on the relative condition numbers derived in  \cref{sec_forward}. To provide a fair comparison, we also consider CG, MINRES (applied to \cref{augmented}) and standard direct methods.  We show that CGLS$I$ and CGLS$\epsilon$ perform better than CG and that they can compare with direct methods, in terms of solution accuracy. 

\subsection{Problem definition and methodology}

All the numerical methods have been implemented in Matlab. For CGLS$\epsilon$\footnote{http://web.stanford.edu/group/SOL/software/cgls/}, CG\footnote{http://www4.ncsu.edu/eos/users/c/ctkelley/www/roots/pcgsol.m} and MINRES\footnote{https://web.stanford.edu/group/SOL/software/minres/}, Matlab codes available online have been employed. For CG, the computation $A^\top A$ is avoided and the products of $A^\top A$ times a vector are computed by successive multiplication by $A$ and its transpose. We consider linear systems of the form \cref{base}, where the matrix $A\in \mathbb{R}^{m\times n}$ has been obtained as $A=U\Sigma V^\top$, where $U$ and $V$ have been selected as orthogonal matrices generated with the Matlab commands {\sf{gallery('orthog',m,j)}}, {\sf{gallery('orthog',n,j)}}\footnote{https://www.mathworks.com/help/matlab/figure/ref/gallery.html}, respectively, for different choices of $j=1,\dots,6$. We consider two choices for the diagonal elements of $\Sigma$:
\begin{itemize}
	\item[C1]: $\Sigma_{ii} = a^{-i}$, for $a>0$,
	\item[C2]: $\Sigma_{ii} = u_i$, where $u \in \mathbb{R}^n$ is generated with the  {\sf{linspace}} Matlab command i.e. $u=\sf{linspace}(\sf{dw,up},n)$, with $\sf{dw,up}$ being strictly positive real values,
\end{itemize}
for $i=1,\dots,n$. 

The numerical tests are intended to show specific properties of the method. We therefore consider matrices of  relatively small dimensions ($m=40$ and $n=20$), in order to avoid too ill-conditioned problems \cite{bjorck1998stability}.
For all the performance profiles reported in the following, we consider a set of $40$ matrices of slightly larger dimension ($m=100$, $n=50$), that we later call $\mathcal{P}$, to also test the robustness of the methods (we however do not focus here with issues related to large scale problems).  This set is composed of selected matrices from the {\sf{gallery}} Matlab command (those with condition number still lower than $10^{10}$), and synthetic matrices corresponding to both choices C1 and C2 (of size $100\times 50$ rather than $40\times 20$). The condition number of the matrices is found to be between $1$ and $10^{10}$. The optimality measure considered is the relative solution accuracy $\displaystyle \frac{\|x-\hat{x}\|}{\|x\|}$, with $x$ the exact solution (chosen to be $x=[n-1,n-2,\dots,0]^\top$) and $\hat{x}$ the computed numerical solution (we omit the subscript $\epsilon$ for CGLS$\epsilon$). In the tests $c$ is  freely chosen and $b=Ax-{A^\dagger}^\top  c$. A simulation is considered unsuccessful if the relative solution accuracy is larger than $10^{-2}$. 

We first focus on iterative methods. Before performing a detailed comparison, we address a practical question in CGLS$\epsilon$: how to choose the parameter $\epsilon$ ?  

\subsection{How to choose $\epsilon$ in CGLS$\epsilon$?}\label{sec_num_eps}
In this section we exclusively focus on CGLS$\epsilon$ and on the choice of parameter $\epsilon$. 

In exact arithmetic, when $\epsilon$ goes to zero, applying CGLS to \cref{syst_eps} is equivalent to applying CG to \cref{base}.  However, in practice a fixed value for $\epsilon$ has to be chosen. This choice should minimize the error on the computed solution $\hat{x}_\epsilon$, which depends on two terms, the distance of the computed solution to the exact solution of \cref{syst_eps} and the distance of the exact solution of \cref{syst_eps} to the exact solution of \cref{base}, as it is shown in relation \cref{err_eps}.

We plot the estimates of these two terms as a function of $\epsilon$ in \cref{fig:eps1}, for the test corresponding to $c=\sf{rand}(n,1)$ and choice C2 with $\sf{dw}=10^{-7}$, $\sf{up}=0.1$. Specifically, on  \cref{fig:eps1} left we plot the estimate \cref{error_estim_eps} (solid line) of $\displaystyle \frac{\|x-\xe\|}{\|x\|}$ for different values of $\epsilon$ ($\epsilon=2^{-27}\approx 10^{-8}, 2^{-34}\approx 10^{-10},2^{-40}\approx 10^{-12},2^{-47}\approx 10^{-14},2^{-54}\approx 10^{-16}$). On  \cref{fig:eps1}  right we plot the relative condition number \cref{rel_cond_eps}, solid line.

With decreasing $\epsilon$, the solution of \cref{syst_eps} approximates increasingly better that of \cref{base}, and the relative condition number improves and converges to that of \cref{base} in \cref{rel_cond} (dotted line in the right part of \cref{fig:eps1}).  A small value of $\epsilon$ will then make both error terms small. 

\vskip 5pt
Special attention should however be paid when the norm of the right-hand side is large. Indeed, in this case the value of $\epsilon$ that is necessary to let the method converge to a solution of \cref{base}, may be really small. 
Let us explain why. 

Let us assume to apply  \cref{algo_cg} to \cref{base} and  \cref{algo_cgls} to \cref{syst_eps}. We would respectively compute:
\begin{align*}
\alpha_1=\frac{\|r_0\|^2}{p_1^\top A^\top Ap_1}=\frac{\|A^\top b+c\|^2}{\|A(A^\top b+c)\|^2}, &&
x_1=\alpha_1(A^\top b+c)=\alpha_1p_1,
\end{align*}
and
\begin{align*}
\alpha_1(\epsilon)=\frac{\|A^\top b+c\|^2}{\|A(A^\top b+c)\|^2+\epsilon\|c^\top(A^\top b+c)\|}, &&
x_1(\epsilon)=\alpha_1(\epsilon)p_1(\epsilon)=\alpha_1(\epsilon)p_1.
\end{align*}

Notice that if $\epsilon$ tends to zero, so does the term $\epsilon\|c^\top(A^\top b+c)\|$ in the denominator of $\alpha_1(\epsilon)$. Consequently $\alpha_1(\epsilon)$ tends toward $\alpha_1$ and $x_1(\epsilon)$ tends toward $x_1$. 
If $\epsilon$ has to be fixed, its value should be small enough to let $\epsilon\|c^\top(A^\top b+c)\|$ be small compared to $\|A(A^\top b+c)\|^2$, otherwise the found approximation will be close to a solution of \cref{syst_eps} rather than to one of \cref{base}.  This choice is then  particularly difficult when $\|A^\top b+c\|$ is large.

We show this phenomenon on a test case reported in  \cref{fig:test_eps}. We take choice C1 with $a=1.3$. It holds $\kappa(A)=10^2$. We first choose $c=10^{-2}\,\sf{rand}(n,1)$ (\cref{fig:test_eps} left) and then $c=10^{2}\,\sf{rand}(n,1)$ (\cref{fig:test_eps} right).  We report the relative error ${\|x-\hat{x}\|}/{\|x\|}$ versus the number of iterations $k$.  The related system has been solved by CGLS$\epsilon$,  for two different values of $\epsilon\approx10^{-12}, 10^{-14}$. We can remark that in both cases the lower value of the parameter allows for a lower relative error but especially in the second test the difference is really important. This is due to the fact that in the second test the norm of $c$ is much larger than in the first test, so that a smaller value of $\epsilon$ would be needed to let the term $\epsilon\|c^\top(A^\top b+c)\|$ be negligible with respect to $\|A(A^\top b+c)\|$.

Finally, in  \cref{fig:eps2}, we report a performance profile \cite{perf_prof}, built on the set of matrices $\mathcal{P}$. The performance of the method improves as $\epsilon$ becomes smaller, both in terms of robustness and accuracy, with the best corresponding to $\epsilon\approx 10^{-14}$, as shown in  \cref{fig:eps2}. Lower values of $\epsilon$ do not improve the performance of the method.

\vskip 5pt

{\bf Conclusion:} A small value of $\epsilon$ is found to guarantee the best performance, and ensures CGLS$\epsilon$ to be less sensible to right-hand sides with large Euclidean norm.  Then, later we choose $\epsilon=2^{-47}\approx 10^{-14}$.

\begin{figure}
	\centering
	\includegraphics[width=0.49\linewidth,height=0.35\linewidth]{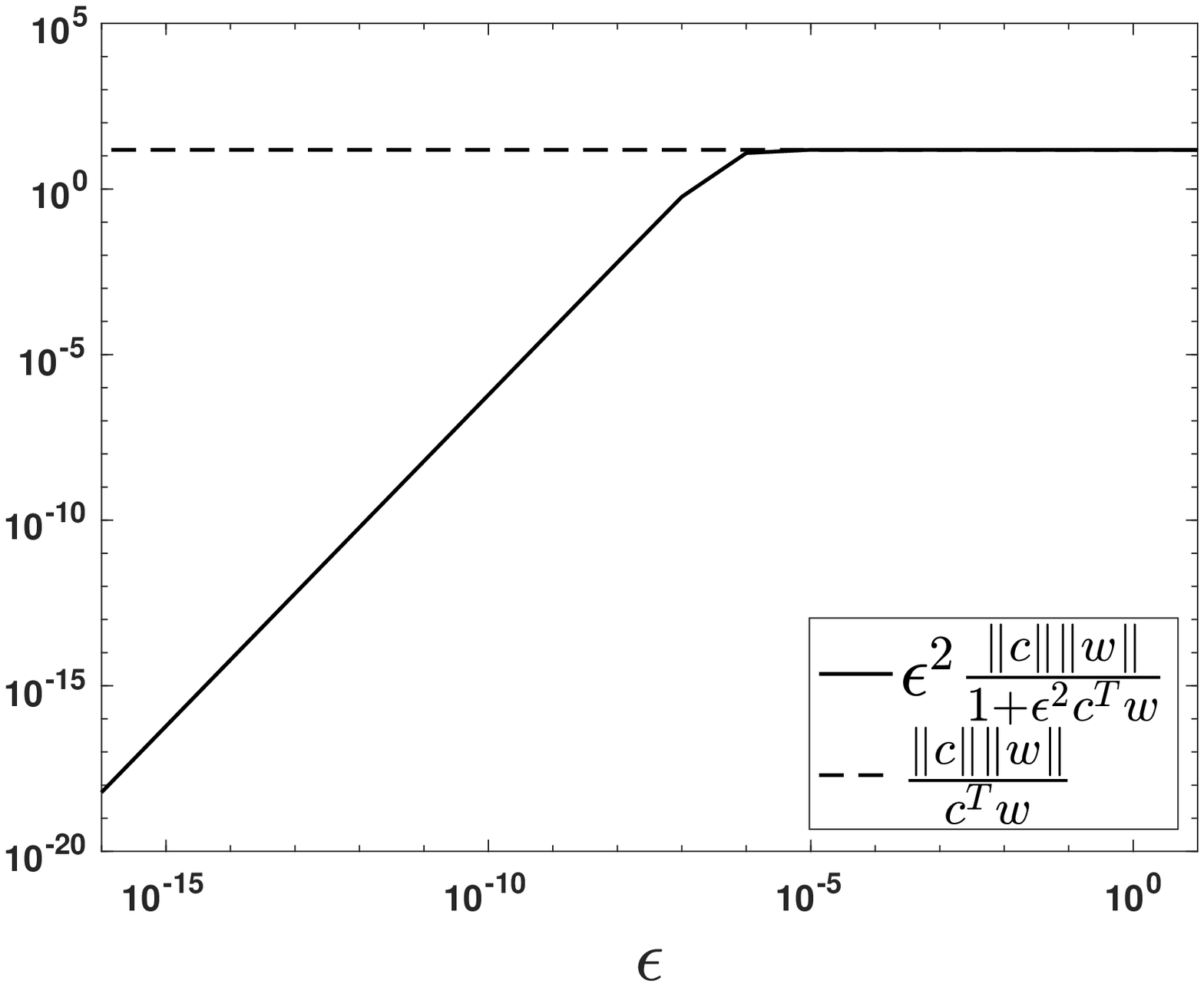}
	\includegraphics[width=0.49\linewidth,height=0.35\linewidth]{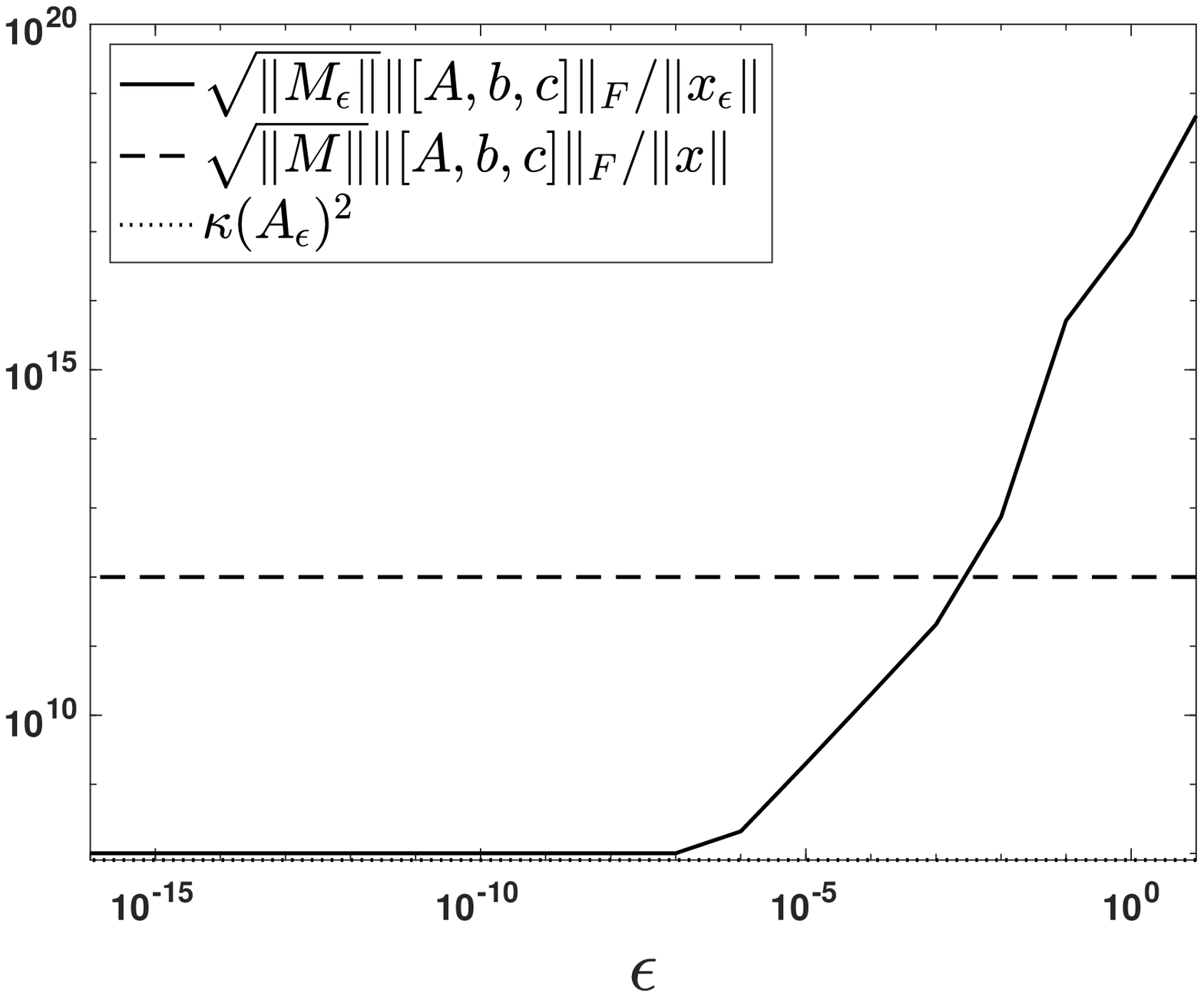}
	\caption{CGLS$\epsilon$. Left part: Estimate \cref{error_estim_eps} of $ \|x_\epsilon-x\|/\|x\|$ for different values of $\epsilon$ (solid line) and limit value for $\epsilon\rightarrow\infty$ (dashed line), $w=(A^\top A)^{-1}c$. Right part: relative condition number \cref{rel_cond_eps}  for different values of $\epsilon$ (solid line) and  relative condition number for $\epsilon=0$, from \cref{rel_cond} (dotted line). Notice that for small values of $\epsilon$ the structured condition number of the problem is far less than $\kappa(A_\epsilon)^2$ (dashed line).}
	\label{fig:eps1}
\end{figure}

\begin{figure}
	\centering
		\includegraphics[width=0.49\linewidth,height=0.35\linewidth]{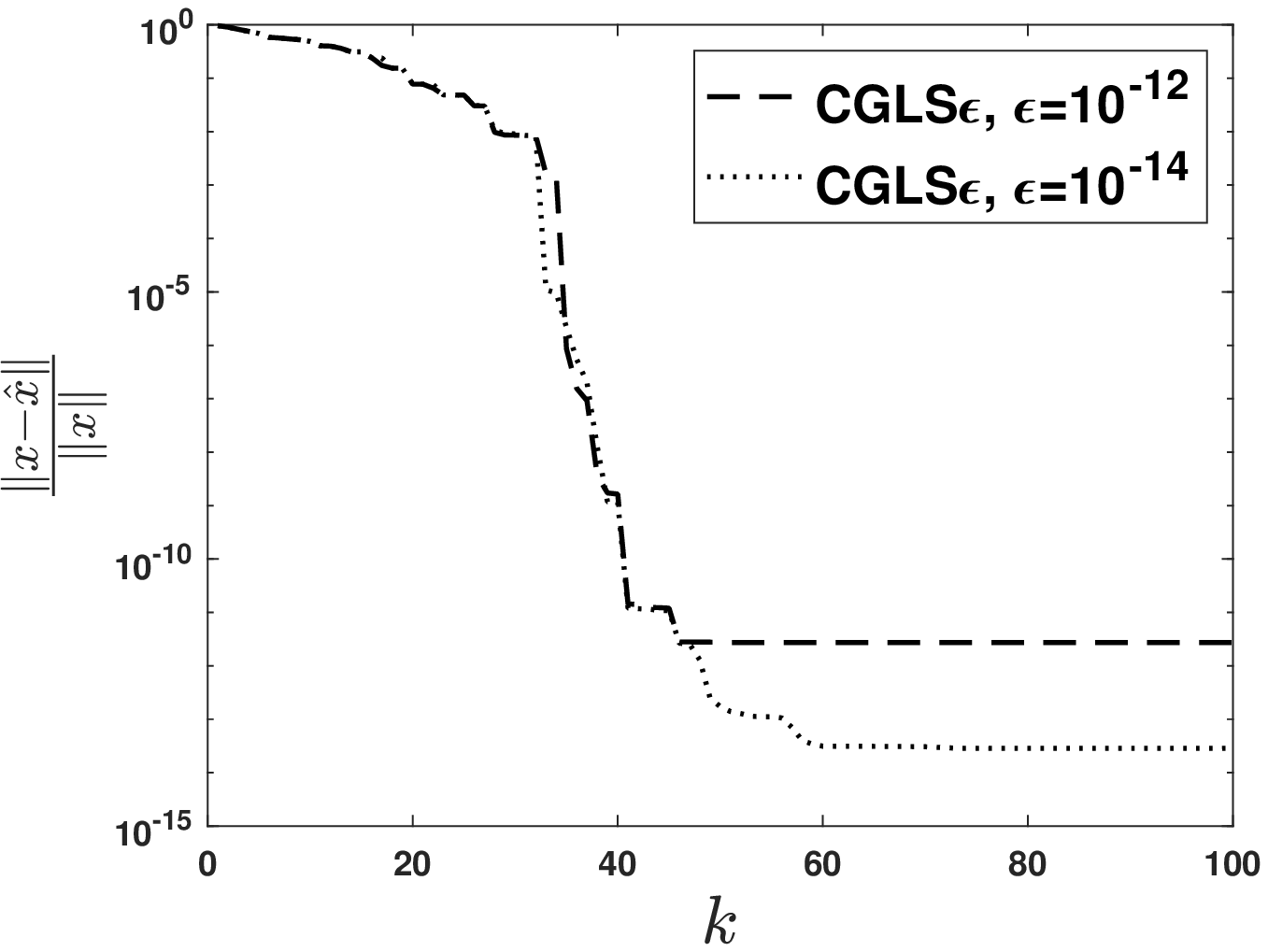}
		\includegraphics[width=0.49\linewidth,height=0.35\linewidth]{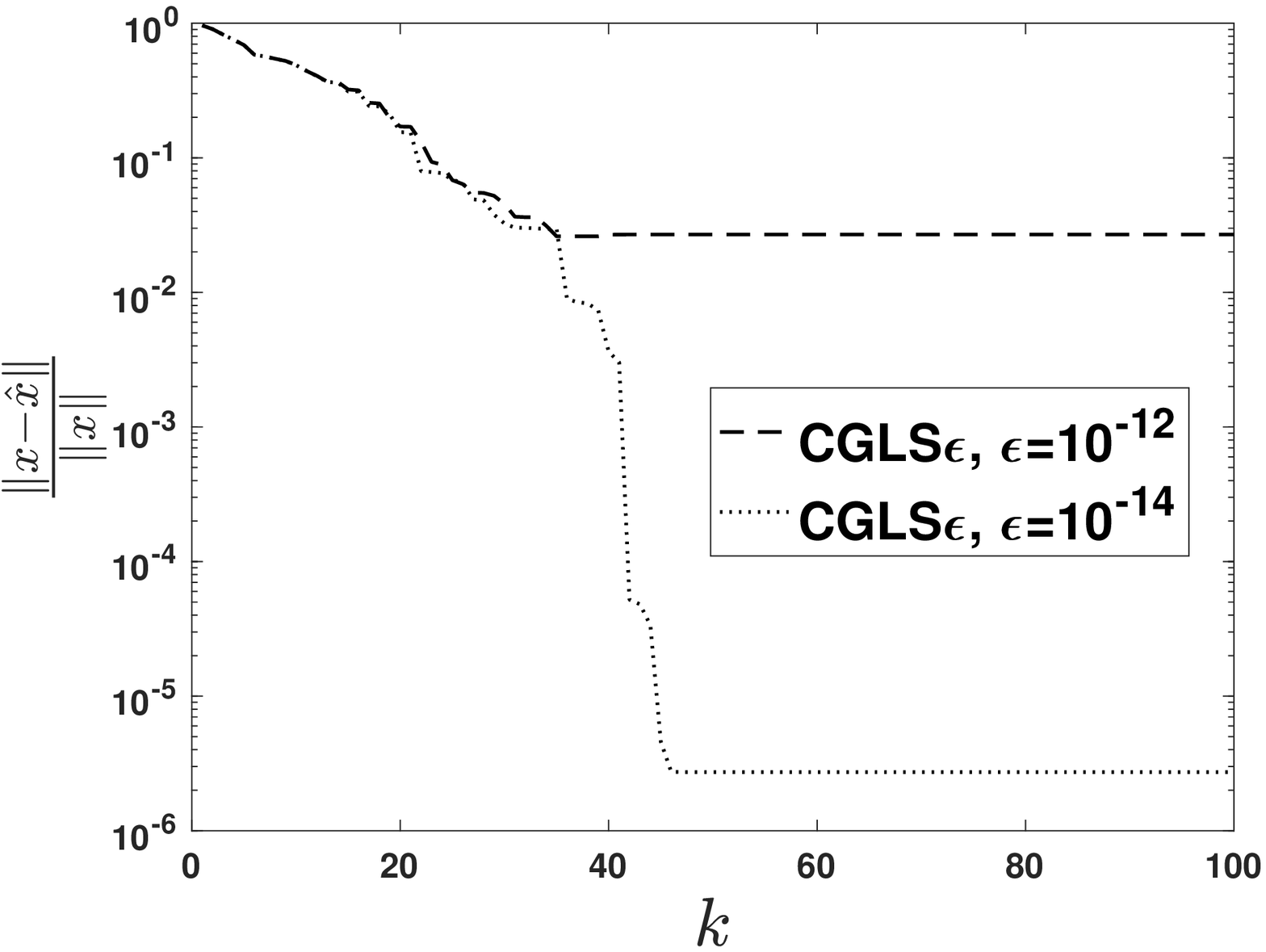}
	\caption{CGLS$\epsilon$. Relative error $\displaystyle \frac{\|x-\hat{x}\|}{\|x\|}$ between the exact solution $x$ and the computed solution $\hat{x}$ for $\epsilon\approx10^{-12}$ (dashed line) and $\epsilon\approx10^{-14}$ (dotted line) versus iteration index $k$. Left part: $c = 10^{-2}\,\sf{rand(n,1)}$, right part: $c = 10^{2} \,\sf{rand(n,1)}$.}
	\label{fig:test_eps}
\end{figure}

\begin{figure}
	\centering
	\includegraphics[scale=0.45]{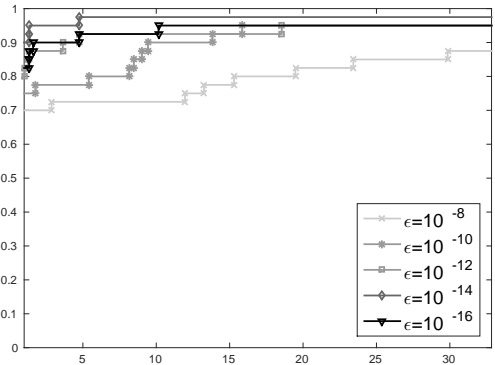}
	\caption{CGLS$\epsilon$. Performance profile of CGLS$\epsilon$ on the synthetic set of matrices $\mathcal{P}$ for different values of $\epsilon$.}
	\label{fig:eps2}
\end{figure}

\subsection{Comparing CGLS$\epsilon$ and CGLS$I$ with CG}\label{sec_comp}
In this section, we compare the performance of CGLS$\epsilon$ (with $\epsilon=2^{-47}\approx10^{-14}$) and CGLS$I$ to the reference CG method. We first show their increased performance on selected problems. 

\subsubsection{Solution accuracy}

We first consider the numerical experiment corresponding to choice C1 with $a=0.5$ and we choose $c=10^{-1}\,\sf{rand}(n,1)$. The corresponding matrix $A$ is such that $\kappa(A)=10^5$. The three iterative methods are compared in  \cref{fig:test1_err} left, where we report the relative error ${\|x-\hat{x}\|}/{\|x\|}$ versus the number of iterations $k$. The practical behaviour of the methods is found to be different.  CG achieves an accuracy of the order of $10^{-7}$, while the error in the solution computed by CGLS$I$ and CGLS$\epsilon$ is $\kappa(A)$ times lower. 

\begin{figure}
	\centering
	\includegraphics[width=0.49\linewidth,height=0.35\linewidth]{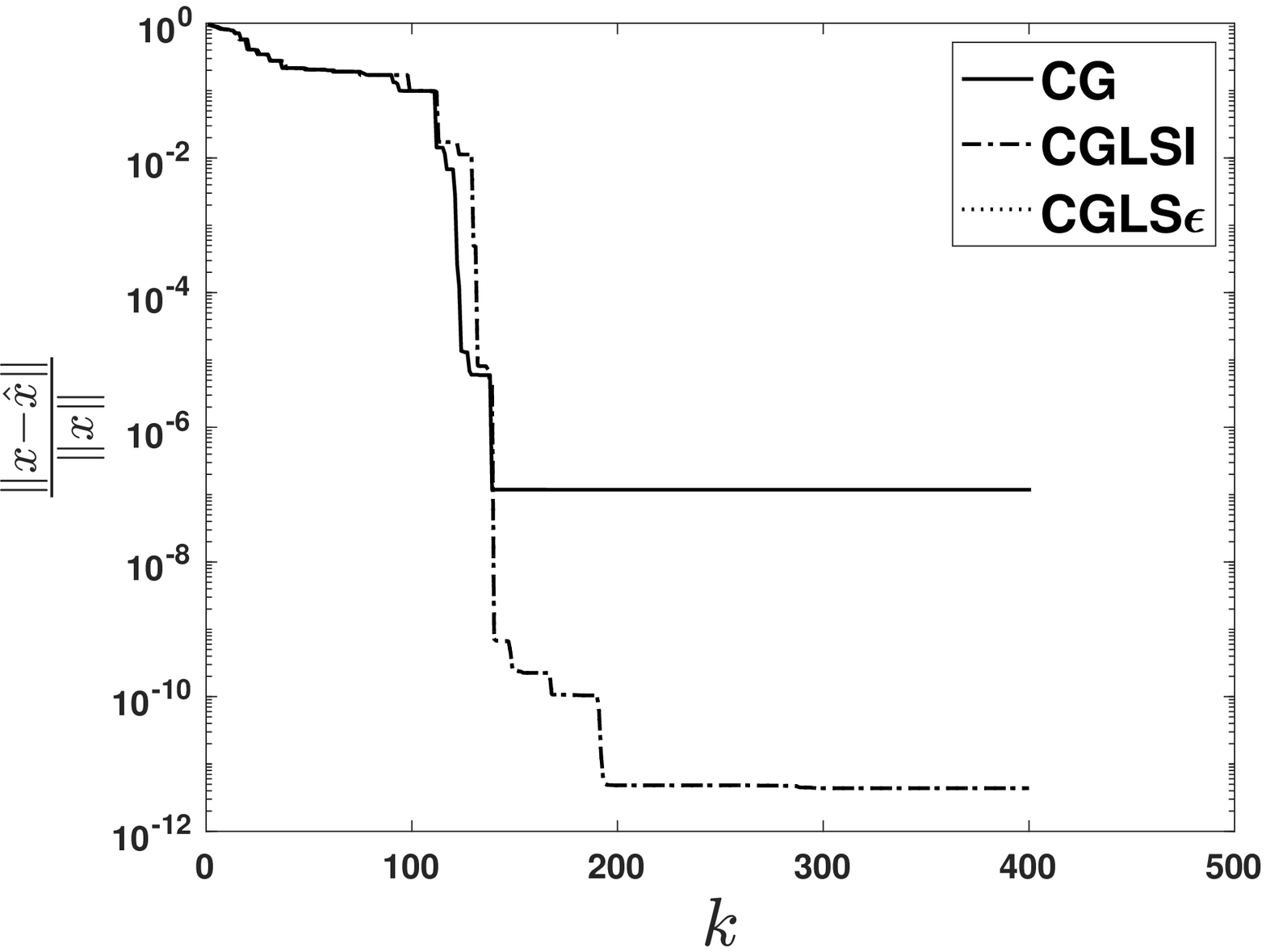}
	\includegraphics[width=0.49\linewidth,height=0.35\linewidth]{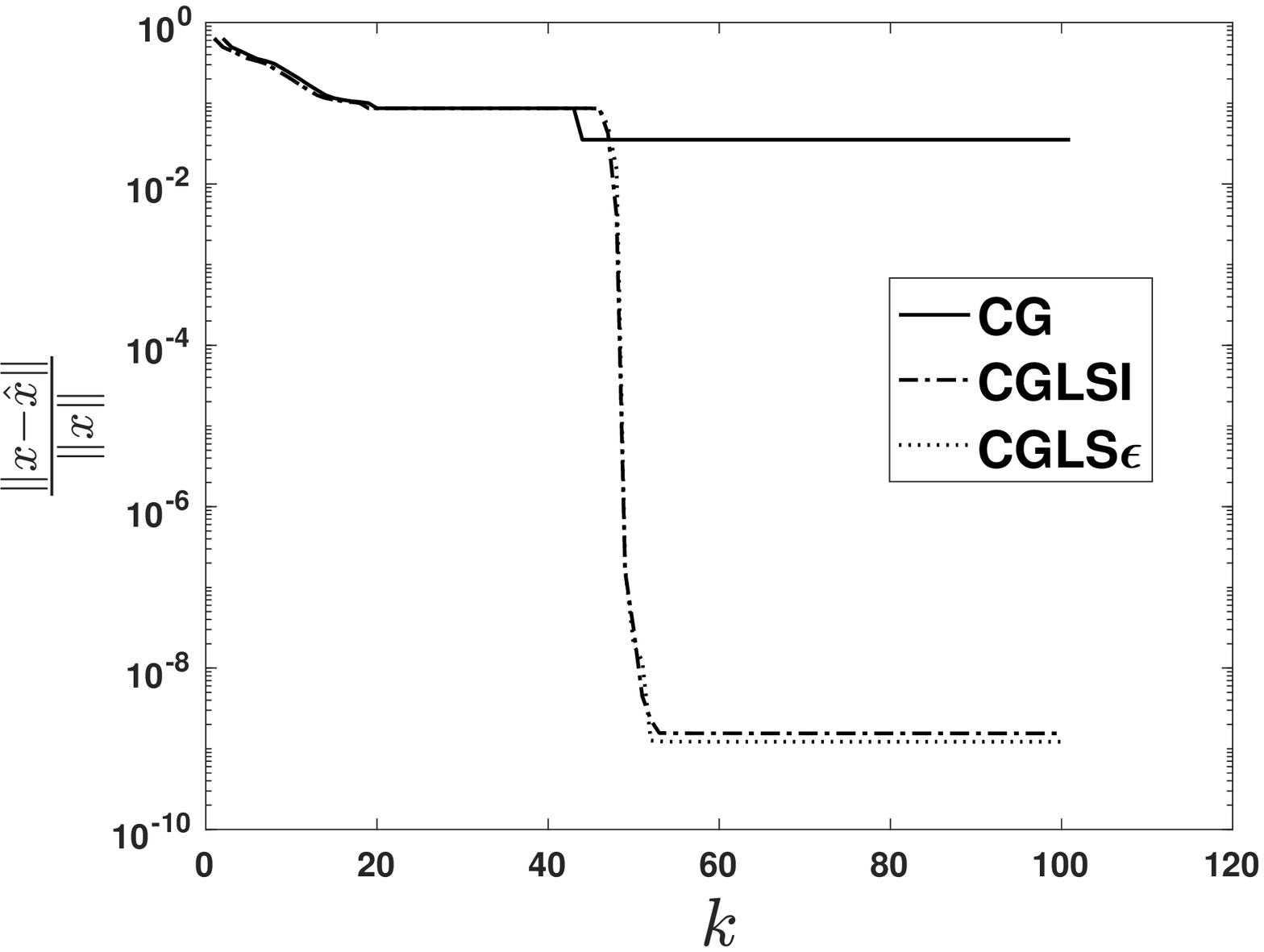}
	\caption{Relative error $\displaystyle \frac{\|x-\hat{x}\|}{\|x\|}$ between the exact solution $x$ and the computed solution $\hat{x}$ for CG, CGLS$I$ and CGLS$\epsilon$, respectively. Left part:  $\kappa(A)=10^5$. Right part:  $\kappa(A)=5\,10^7.$}
	\label{fig:test1_err}
\end{figure}

We now consider a second numerical experiment based on the choice C2 with $\sf{up}=0.5$, $\sf{dw}=10^{-8}$, $c=10^{-14}\,\sf{rand}(n,1)$ respectively. It holds $\kappa(A)=5\,10^7$. The gap in the results is even larger, both CGLS$I$ and CGLS$\epsilon$ find an accurate solution, while CG does not manage to solve the problem at all, see  \cref{fig:test1_err} right. 

We finally compare the three iterative methods plus MINRES on the synthetic set of matrices $\mathcal{P}$, where $c$ is chosen as $\gamma+(\zeta-\gamma)\,\sf{rand}(n,1)$  for different values of $\gamma, \zeta$. We report in  \cref{fig:perf_iter} the performance profile corresponding to these simulations. Clearly, CGLS$I$ and CGLS$\epsilon$ perform much better than CG and MINRES. Their performance is quite similar but CGLS$I$ is found more robust on problems with right-hand sides of large Euclidean norm. In addition, CGLS$I$ offers the advantage to be parameter free. 

\begin{figure}
	\centering	
	\includegraphics[scale=0.5]{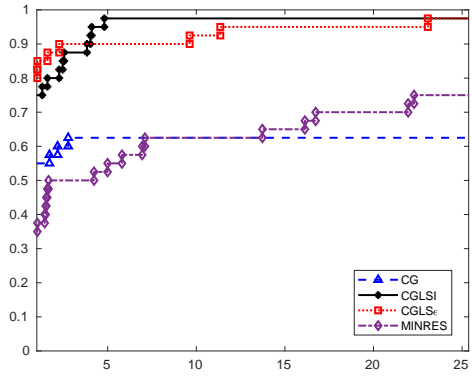}
	\caption{Performance profile of CG, CGLS$I$, CGLS$\epsilon$ and MINRES on the synthetic set of matrices $\mathcal{P}$.}
	\label{fig:perf_iter}
\end{figure}

\subsubsection{Forward error bounds}\label{sec_num_bounds}
In this section we numerically validate the  error estimates   \cref{cond}.

In  \cref{tab:tab} we validate these bounds for different problems of the form C1 and C2. We report the condition number of matrix $A$, the bound on the forward error coming from standard theory of linear systems and, for each method, the forward error and the estimates from \cref{cond}.  For all the tests, the proposed first order estimates provide an accurate upper bound for the forward error and in many tests they are much sharper than the standard bound. It is again evident from these tests that in general CGLS$I$ and CGLS$\epsilon$ perform better than CG. In particular, in many tests the initial error introduced while forming the right-hand side, see \cref{bound}, is larger than the error related to the solution of the linear system.

Notice also that using the standard results on least squares problems for CGLS$\epsilon$  would largely overestimate the error. For example, for the first test we would obtain:
\begin{equation*}
\frac{\|\hat{x}_\epsilon-x_\epsilon\|}{\|x_\epsilon\|}\leq u~\kappa(\Ae)\left(1+\frac{\|\Ae^\dagger\|\|b_\epsilon-\Ae x_\epsilon\|}{\|x_\epsilon\|}\right)\approx 3\, 10^{-8}.
\end{equation*}
This is evident also from  \cref{fig:eps1} right, where we remark that $\kappa^2(A_{\epsilon})$ alone is already much larger than the structured relative condition number of \cref{sistema_eps}.


	\begin{table}
		\centering
		\begin{tabular}{cclllllll}
			\hline
			Problem	& $\kappa(A)$& $\kappa(A)^2\bar{\eta}(\hat{x})$ & $E_{CG}$ & $\hat{E}_{CG}$ & $E_{CGLSI}$ & $\hat{E}_{CGLSI}$ & $E_{CGLS\epsilon}$ & $\hat{E}_{CGLS\epsilon}$ \\ 
			\hline
			$a=2$, $\alpha=10^{-10}$	&$5\,10^{5}$&$6\,10^{-5}$&$ 6\,10^{-7}$ & $10^{-5}$ &$2\,10^{-10}$ & $3\,10^{-6}$ &$2\,10^{-10}$ &$4\,10^{-7}$ \\ 
			$a=0.4$, $\alpha=10^{-12}$	&$10^{7}$&$10^{-1}$&$10^{-3}$  & $10^{-2}$ & $10^{-8}$ & $4\,10^{-8}$ & $7\,10^{-11}$ &$5\,10^{-8}$\\ 
			$a=0.7$, $\alpha=10^{-1}$ &$10^{3}$&$ 2\,10^{-10}$ & $2\,10^{-12}$ & $10^{-11}$ & $5\,10^{-15}$ &$2\,10^{-12}$ &$5\,10^{-15}$ &$3\,10^{-13}$\\ 
			$a=1.3$, $\alpha=10^{-4}$ &$10^{2}$&$ 5\,10^{-12}$ & $2\,10^{-13}$ & $10^{-12}$ & $2\,10^{-15}$ &$\,10^{-12}$ &$\,10^{-14}$ &$\,10^{-13}$\\ 
			$\sf{up}=10^{2}$, $\sf{dw}=10^{-4}$, $\alpha=10^{-4}$	&$10^{6}$&$2\,10^{-4}$  & $10^{-5}$ & $10^{-4}$ & $10^{-10}$ & $10^{-6}$ &$10^{-10}$&$10^{-7}$\\ 
			$\sf{up}=10^{-2}$, $\sf{dw}=10^{-6}$, $\alpha=10^{-5}$	&$\,10^{4}$&$8\,10^{-7}$  & $2\,10^{-9}$ & $7\,10^{-7}$ & $5\,10^{-9}$ &$7\,10^{-7}$  &$8\,10^{-10}$&$\,10^{-7}$\\
			$a=1.9$, $\alpha=-10^{-6}$&$2\,10^{5}$&$10^{-5}$  & $4\,10^{-8}$ & $10^{-6}$ & $3\,10^{-9}$ & $\,10^{-6}$ &$3\,10^{-9}$&$6\,10^{-7}$\\ 
			$\sf{up}=10^{3}$, $\sf{dw}=10^{-1}$, $\alpha=10^{2}$&$10^{4}$	&$ 4\, 10^{-5}$ & $10^{-9}$ & $10^{-6}$ & $6\,10^{-15}$ & $10^{-6}$& $10^{-11}$& $10^{-7}$\\ 
			$\sf{up}=10^{4}$, $\sf{dw}=10^{-3}$, $\alpha=-10^{-2}$ &$10^{7}$&$10^{-2}$  & $10^{-3}$ & $10^{-2}$ & $10^{-9}$ & $10^{-6}$ &$10^{-9}$ &$10^{-7}$\\
			$a=0.5$, $\alpha=1$ &$10^{5}$&$ 10^{-5}$ &$ 10^{-7}$ & $8\,10^{-6}$ & $5\,10^{-12}$ & $8\,10^{-10}$ &$5\,10^{-12}$ &$2\,10^{-10}$ \\ 
			\hline
		\end{tabular}
	\caption{Comparison of the computed forward errors and their estimates  for different synthetic test problems. We report the condition number $\kappa(A)$ of $A$, the standard bound $\kappa(A)^2\bar{\eta}(\hat{x})$, the quantities $E_{CG}$, $E_{CGLSI}$,  and $E_{CGLS\epsilon}$ which denote the computed forward errors $\displaystyle \frac{\|x-\hat{x}\|}{\|x\|}$ for  CG, CGLS$I$ and CGLS$\epsilon$, respectively and the corresponding estimates $\hat{E}_{CG}$, $\hat{E}_{CGLSI}$ and $\hat{E}_{CGLS\epsilon}$  from \cref{cond}. In the tests $c=\alpha\,\sf{rand}(n,1)$ for a chosen constant $\alpha$.}

	\label{tab:tab}
\end{table}


\begin{figure}
	\centering
	\includegraphics[scale=0.4]{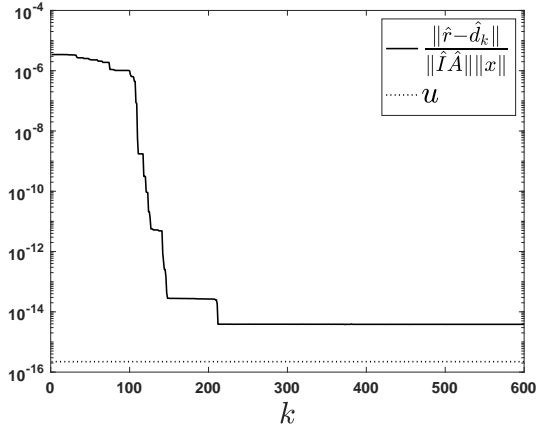}
	\caption{CGLS$I$.  $\displaystyle \frac{\|\hat{r}-\hat{d}_k\|}{\|\hat{I}\hat{A}\|\|x\|}$ with $\hat{r}=\hat{b}-\hat{I}\hat{A}x$ and $\hat{d}_k$ the recurred residual defined in  \cref{algo_i} versus iteration index $k$ and machine precision $u\approx10^{-16}$. }
	\label{fig:fig3}
\end{figure}

\subsubsection{Numerical study of the norm of the residual}
In this section we study the numerical behaviour of the Euclidean norm of the residuals. 

In  \cref{fig:fig3}, we consider for CGLS$I$ the quantity $\displaystyle \frac{\|\hat{r}-\hat{d}_k\|}{\|\hat{I}\hat{A}\|\|x\|}$ (with $\hat{r}=\hat{b}-\hat{I}\hat{A}x$ and $\hat{d}_k$ defined in  \cref{algo_i}) that appears in \cref{cond_bjork}. We can deduce that the bound \cref{bound_error} holds for $k$ large enough, as \cref{cond_bjork} is satisfied with $c_1=O(10)$, see \cref{remark_bjork}. This is also the case in many other simulations. 

\vskip 5pt
{\bf Conclusion:} The proposed methods show better performance that standard CG, both in terms of accuracy and of rate of convergence. The error bounds proposed, based on structured condition numbers of the problems, are able to better predict forward errors than classical bounds. The relative difference between the actual and recurred residuals in CGLS$I$ is of the order of the machine precision.

\subsection{Final comparison with direct methods}\label{sec_direct}
Motivated by the applications introduced in the Introduction, we have mainly focused on the design of iterative methods. However, it is natural to consider also direct methods for the solution of problem \cref{base}. It is then interesting to compare the performance of our methods to direct approaches known to be backward stable. 

A first possibility is to solve directly \cref{base}.  We consider the method, later denoted by QR, that expresses $A^\top A$ as $R^\top R$, where the upper triangular factor $R$ is obtained from the $QR$ factorization of $A$. 

Alternatively, we can consider \cref{sistema_eps} and use a method that computes the QR factorization of $\Ae$. This can be computed directly or  obtained as a modification of the QR factorization of $A$, to take into account the addition of the last line, as described in  \cite[Section 3.2]{bjorck1996numerical}. For the QR factorization of $\Ae$ column pivoting is used for improved stability, as suggested in \cite[Chap.22]{laha:74} for the solution of linear least squares problems with linear equality constraints by weighting. No evident difference was observed in the numerical tests between the two methods, so just results obtained by the  QR factorization of $\Ae$ are reported in the following. The method will be labelled as $QR_\epsilon$.

Another meaningful possibility is to solve the normal equations \cref{syst_eps}, computing the inverse of $A^\top A+\epsilon^2 c c^\top$ with the Sherman-Morrison formula \cite[\S 3.1]{bjorck1996numerical}:
\begin{equation*}
(A^\top A+\epsilon^2cc^\top)^{-1}=(A^\top A)^{-1}-\alpha ww^\top, \quad w=(A^\top A)^{-1}c,\quad \alpha=\frac{\epsilon^2}{1+\epsilon^2c^\top w}.
\end{equation*} 
In this case, we can express the solution $x$ as:
\begin{equation*}
x=(I_n-\alpha wc^\top)(x^\dagger+w), \quad x^\dagger=A^\dagger b.
\end{equation*}
We label this method as SM. 

Finally, we can consider the reformulation of \cref{base} as the augmented system \cref{augmented} and use AUG that solves the system with a $LDL^\top$ factorization, after having applied the optimal scaling described in \cite{bjor92}. All the considered methods are summarized in \cref{tab:tab_methods}.

	\begin{table}
		\centering
		\caption{Summary of all the considered methods (both direct and iterative). We report the label used for the method, the related formulation of the problem the method is applied to and a brief description of each method. See \cref{sec_methods} for additional details on iterative methods.}
		\begin{tabular}{l|l|l}
			\hline
			Label &  Formulation & Description\\
			\hline
			\rule[0mm]{0mm}{4mm} QR & $A^\top Ax=A^\top b+c$  &  QR factorization of $A$\\
			\rule[0mm]{0mm}{4mm} QR$_\epsilon$  &$\min_{x}\|A_\epsilon x-b_\epsilon\|^2$& QR factorization of matrix $\Ae$\\
			\rule[0mm]{0mm}{4mm} SM& $(A^\top A+\epsilon^2 cc^\top)x=A^\top b+c$ & \small{Inverse  with Sherman-Morrison formula}\\
			\rule[-4mm]{0mm}{0mm} AUG &  $\begin{bmatrix}
			I_m & A\\
			A^\top& 0
			\end{bmatrix}
			\begin{bmatrix}
			r\\x
			\end{bmatrix}=\begin{bmatrix}
			b\\-c
			\end{bmatrix}$ & $LDL^\top$ factorization after scaling\\
			\hline
			\rule[0mm]{0mm}{4mm} CG & $A^\top Ax=A^\top b+c$ & Conjugate Gradient\\
			\rule[0mm]{0mm}{4mm} CGLS$I$ & $\hat{A}^\top\hat{I}\hat{A}x=\hat{A}^\top\hat{b}$ & Modified Conjugate Gradient\\
			\rule[0mm]{0mm}{4mm} CGLS$\epsilon$ & $\min_{x}\|A_\epsilon x-b_\epsilon\|^2$ & Conjugate Gradient for Least Squares\\
			\rule[-4mm]{0mm}{0mm} MINRES &  $\begin{bmatrix}
			I_m & A\\
			A^\top& 0
			\end{bmatrix}
			\begin{bmatrix}
			r\\x
			\end{bmatrix}=\begin{bmatrix}
			b\\-c
			\end{bmatrix}$ & Minimum Residual method
			\\
			\hline
		\end{tabular}	
		\label{tab:tab_methods}
	\end{table}

\begin{figure}
	\centering

	\includegraphics[scale=0.5]{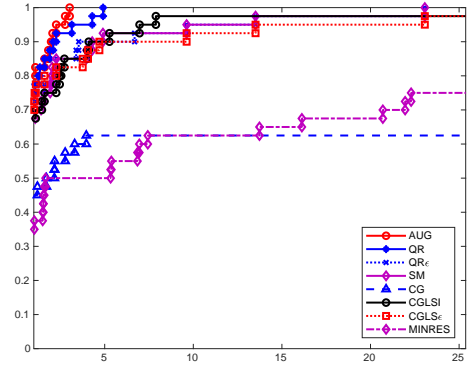}
	\caption{Performance profile on the synthetic set of matrices 
		$\mathcal{P}$  considering both direct and iterative methods. 
		 }
	\label{fig:perf_tutti}
\end{figure}
\vspace{5pt}
In  \cref{fig:perf_tutti} we compare all these direct methods with the iterative methods on the same set of matrices used for the performance profile in  \cref{fig:perf_iter}. It is clear that CG and MINRES show a much worse performance than the other solvers. On the other hand, even if AUG and QR remain slightly more robust and more efficient,  the performance of the proposed iterative methods is close to that of direct backward stable methods. 


\vskip 5pt 

{\bf Conclusion:} The overall performance of the proposed iterative methods is close to that of standard direct backward stable methods.

\section{Conclusions} \label{sec:conc}
We have considered both theoretical and practical aspects related to the solution of linear systems of the form \cref{base}. First, we have studied the structured condition number of the system and proposed a related explicit formula for its computation. Then, we have considered the numerical solution of \cref{base}. We have remarked that the same issues that deteriorate the performance of the Conjugate Gradient method on the normal equations also arise in this setting. This has guided us in the development of robust iterative methods for the numerical solution of \cref{base}. The new methods that have been proposed have been validated numerically. 
\vskip 5pt
From the numerical experiments, we can draw the following conclusions: 
\begin{itemize}
	\item The choice of the $\epsilon$ parameter in CGLS$\epsilon$ deeply affects the performance of the method. A small value (approximately of order $10^{-14}$) is generally found a relevant choice, but a special attention is required in case of right-hand sides with large Euclidean norm.
	\item The proposed iterative methods are shown to be more stable than CG on the set of test problems considered. 
	\item The first order estimates of the forward error obtained by the relative structured condition number are sharper estimates of the forward error, compared to standard bounds issued from linear system theory.
	\item The performance of the proposed methods, in terms of solution accuracy, is comparable to that of stable direct methods.
	\item For the considered class of linear systems, the best iterative solver appears to be CGLS$I$, that, compared to CGLS$\epsilon$, offers the advantage of being parameter free and more robust when solving problems exhibiting right-hand sides with large Euclidean norm. 
\end{itemize}   

\bibliographystyle{siamplain}
\bibliography{article_simax}


\end{document}